\documentclass[11pt,english,reqno]{smfart}

\usepackage{babel}
\usepackage{amscd}
\usepackage{amssymb,amsthm,amsmath,amsfonts}
\usepackage{hyperref}
\usepackage[all]{xy}
\usepackage{mathtools}

\def\O{{\mathcal O}}
\def\X{{\mathcal X}}
\def\sX{{\mathcal X}}
\def\Y{{\mathcal Y}}
\def\U{{\mathcal U}}
\def\V{{\mathcal V}}
\def\W{{\mathcal W}}
\def\sY{{\mathcal Y}}
\def\sD{{\mathcal D}}
\def\sE{{\mathcal E}}
\def\sZ{{\mathcal Z}}
\def\sP{{\mathcal P}}
\def\sQ{{\mathcal Q}}
\def\sR{{\mathcal R}}

\def\F{{\mathbf F}}

\def\ov{\overline}
\def\lra{\longrightarrow}
\def\val{{\rm val}}

\def\f{{\rm f}}

\numberwithin{equation}{section} \numberwithin{figure}{section}

\DeclareMathOperator{\Gal}{Gal}

\DeclareMathOperator{\Spec}{Spec}

\DeclareMathOperator{\Hom}{Hom}

 \DeclareMathOperator{\Norm}{N}
\DeclareMathOperator{\Frob}{Frob} \DeclareMathOperator{\dens}{dens}
\DeclareMathOperator{\Irr}{Irr} \DeclareMathOperator{\Res}{R}
\DeclareMathOperator{\Cl}{Cl} \DeclareMathOperator{\HH}{H}

\newcommand\FF{\mathbf{F}}
\newcommand\PP{\mathbf{P}}

\newcommand\ZZ{\mathbf{Z}}
\newcommand\NN{\mathbf{N}}
\newcommand\QQ{\mathbf{Q}}

\newcommand\CC{\mathbf{C}}
\newcommand\GG{\mathbf{G}}

\newcommand\Gm{\GG_\mathrm{m}}
\newcommand\OO{\mathcal{O}}

\newtheorem{theorem}{Theorem}[section]
\newtheorem{corollary}[theorem]{Corollary}
\newtheorem{lemma}[theorem]{Lemma}
\newtheorem{proposition}[theorem]{Proposition}

\theoremstyle{definition}
\newtheorem{definition}[theorem]{Definition}
\newtheorem{example}[theorem]{Example}
\newtheorem{remark}[theorem]{Remark}

\usepackage[usenames, dvipsnames]{color}

\title[Pseudo-split fibres and arithmetic surjectivity]{Pseudo-split fibres and arithmetic surjectivity}
\author{Daniel Loughran}
\address{School of Mathematics, University of Manchester, Oxford Road, Manchester, M13 9PL, United Kingdom}
\email{daniel.loughran@manchester.ac.uk}

\author{Alexei N. Skorobogatov}
\address{Imperial College London, Department of Mathematics, London SW7 2AZ, United Kingdom \emph{and} Institute for the Information Transmission Problems, Russian Academy of Sciences, 19 Bolshoi Karetnyi, Moscow, 127994 Russia}
\email{a.skorobogatov@imperial.ac.uk}

\author{Arne Smeets}
\address{Radboud Universiteit Nijmegen, Heyendaalseweg 135, 6525 AJ Nijmegen, the Netherlands \emph{and} KU Leuven, Departement Wiskunde, Celestijnenlaan 200B, 3001 Heverlee, Belgium} 
\email{arnesmeets@gmail.com}

\begin{document}

\begin{abstract} Let $f: X \to Y$ be a dominant morphism of smooth, proper and geometrically integral varieties over a number field $k$, with geometrically integral generic fibre. We give a necessary and sufficient geometric criterion for the induced map $X(k_v) \to Y(k_v)$ to be surjective for almost all places $v$ of $k$. This generalises a result of Denef which had previously been conjectured by Colliot-Th\'el\`ene, and can be seen as an optimal geometric version of the celebrated Ax--Kochen theorem. \vskip 10pt
Soit $f: X \to Y$ un morphisme dominant de vari\'et\'es lisses, propres et g\'eom\'etriquement int\`egres d\'efinies sur un corps de nombres $k$, dont la fibre g\'en\'erique est g\'eom\'etriquement int\`egre. Nous donnons un crit\`ere g\'eometrique, \`a la fois n\'ecessaire et suffisant, pour que l'application induite $X(k_v) \to Y(k_v)$  soit surjective pour presque toute place $v$ de $k$. Ceci g\'en\'eralise un r\'esultat de Denef pr\'ec\'edemment conjectur\'e par Colliot-Th\'el\`ene. Notre r\'esultat peut \^etre vu comme une version g\'eom\'etrique optimale du c\'el\`ebre th\'eor\`eme de Ax--Kochen.
\end{abstract}

\maketitle

\tableofcontents

\section{Introduction}

\subsection{} 
A famous theorem of Ax--Kochen \cite{AxKochen} states that any homogenous polynomial over $\QQ_p$
of degree $d$ in at least $d^2 + 1$ variables has a non-trivial zero, provided that $p$ avoids a certain finite
exceptional set of primes depending only on $d$. This was originally proved using model theory. Denef recently found  purely algebro-geometric proofs \cite{Denefbis, Denef}. In \cite{Denef}, he did so by proving a more general conjecture of Colliot-Th\'el\`ene \cite[\S 3, Conjecture]{CTCRAS}. 

The essential notion (first introduced by the second author in \cite[Definition~0.1]{Sko96}) appearing in this conjecture is that of a \emph{split scheme}:


\begin{definition} \label{def:split}
Let $k$ be a perfect field. A scheme $X$ of finite type over $k$ is called \emph{split} if $X$ contains an irreducible component of multiplicity $1$ which is geometrically irreducible.
\end{definition}

Here the \emph{multiplicity} of an irreducible component $Z$ of $X$  is the length of the local ring of $X$ at the generic point of $Z$. 
In particular, it has multiplicity $1$ if and only if it is generically reduced. Denef's result \cite[Theorem~1.2]{Denef} is the following.

\begin{theorem}[Denef] \label{deneftheorem}
Let $f: X \to Y$ be a dominant morphism of smooth, proper, geometrically integral varieties over a number field $k$, with geometrically integral generic fibre. Assume that for every modification $f': X' \to Y'$ of $f$ with $X'$ and $Y'$ smooth
such that the generic fibres of $f$ and $f'$ are isomorphic, the fibre $(f')^{-1}(D)$ is a split $\kappa(D)$-variety for every $D~\in~(Y')^{(1)}$.

Then $Y(k_v)=f(X(k_v))$ for all but finitely many places $v$ of $k$.
\end{theorem}

Here $k_v$ denotes the completion of $k$ at the place $v$, $(Y')^{(1)}$ denotes the set of points of codimension $1$ in $Y'$, and $\kappa(D)$ is the residue field of $D$. A \emph{modification} of $f$ is a commutative diagram
\begin{equation} \label{def:modification}
\begin{split}
\xymatrix{
X'\ar[r]^{\alpha_X}\ar[d]_{f'}& X\ar[d]^f\\
Y'\ar[r]^{\alpha_Y}& Y}
\end{split}
\end{equation}
where $f':X'\to Y'$ is a dominant morphism of proper and geometrically integral
varieties over $k$, and 
$\alpha_X: X' \to X$ and $\alpha_Y: Y' \to Y$ are birational morphisms.

One obtains the Ax--Kochen theorem by applying Theorem \ref{deneftheorem} to the universal
family of all hypersurfaces of degree $d$ in $\PP^n$ with $n \geq d^2$;
that the hypotheses of the theorem are satisfied in this
case was shown by Colliot-Th\'{e}l\`{e}ne (see \cite[Remarque~4]{CTCRAS}).

\subsection{} 
In this paper we strengthen Denef's result, by determining conditions which are both \emph{necessary and sufficient}
for the map $f:X(k_v)\to Y(k_v)$ to be surjective for almost all places $v$. Our result uses the following
weakening of Definition \ref{def:split} (in \S\ref{sec:psv} we also give a more general definition over arbitrary ground fields).

\begin{definition}
Let $k$ be a perfect field with algebraic closure $\bar{k}$. A scheme $X$ of finite type over $k$ is called
\emph{pseudo-split} if every element of $\mathrm{Gal}(\bar{k}/k)$ fixes some irreducible component of $X \times_k \bar{k}$ of multiplicity~$1$. \label{pseudosplitdefinition1}
\end{definition}

It is clear that pseudo-splitness is weaker than splitness, the latter meaning that a \emph{single} irreducible component of $X \times_k \bar k$ of multiplicity $1$ is fixed by \emph{all} of $\mathrm{Gal}(\bar k/k)$. With this terminology, we can state our generalisation of Denef's result as follows:

\begin{theorem} \label{maintheorem} 
Let $k$ be a number field. Let $f: X \to Y$ be a dominant morphism of smooth, proper, geometrically integral varieties over $k$ with geometrically integral generic fibre. Then $Y(k_v)=f(X(k_v))$ for all but finitely many places $v$ of $k$ if and only if for every modification $f': X' \to Y'$ of $f$, with $X'$ and $Y'$ smooth, and for every point $D \in (Y')^{(1)}$, the fibre $(f')^{-1}(D)$ is a pseudo-split $\kappa(D)$-variety.
\end{theorem}

In the notation introduced by the first and third named authors in their recent work \cite[\S 3]{LS}, the morphisms $f: X \to Y$ satisfying the conclusion of the theorem are exactly the morphisms such that $\Delta(f') = 0$ for every modification $f'$ of $f$.

\subsection{} 
Theorem \ref{maintheorem} will be deduced from finer results. With $f: X \to Y$ as in Theorem~\ref{deneftheorem}, Colliot-Th\'el\`ene asked in \cite[\S 13.1]{CIME} how the geometry of $f$ relates to the surjectivity of the map $X(k_v) \to Y(k_v)$, for a possibly infinite collection of places $v$. He called this phenomenon ``surjectivit\'{e} arithm\'etique'' (note that this is different from the notion of arithmetic surjectivity studied in \cite{GHMS04}).  We develop general criteria which allow one to decide whether, for an \emph{individual} (but large) place $v$, the map $X(k_v) \to Y(k_v)$ is surjective. They involve certain invariants which we call ``$s$-invariants'', defined in \S \ref{sec:splittingdensities} --  local versions of the $\delta$-invariants introduced in \cite[\S 3]{LS}; their definition is given in terms of the geometry of $f$  and does not involve model theory. 

The following result is proved in \S \ref{sec:surjectivity} using tools from logarithmic geometry, in particular, a logarithmic version of Hensel's lemma and ``weak toroidalisation''.
It should be viewed as the main theorem of the paper and is a geometric criterion, in the style
of Colliot-Th\'{e}l\`{e}ne's conjecture, for surjectivity of the map $X(k_v) \to Y(k_v)$.

\begin{theorem}  \label{thm:CT}
Let $k$ be a number field. Let $f: X \to Y$ be a dominant morphism of smooth, proper, geometrically integral varieties over $k$, with geometrically integral generic fibre. Then there exist a modification $f': X' \to Y'$ of $f$ with $X'$ and $Y'$ smooth, and a finite set of places $S$ of $k$ such that for all $v \notin S$ the following are equivalent:
\begin{enumerate}
\item[(1)] The map $X(k_v) \to Y(k_v)$ is surjective;
\item[(2)] for every codimension $1$ point $D' \in (Y')^{(1)}$, we have $s_{f',D'}(v) = 1$.
\end{enumerate}
\end{theorem}

The invariants $s_{f',D'}(v)$ appearing in the statement will be defined in \S \ref{sec:splittingdensities}. They are defined in terms of the Galois action on the irreducible components of the fibre of $f'$ over $D'$.
One benefit of our approach is that it yields a single model for $f$ which can be used to test arithmetic surjectivity using a finite list of criteria.

A simple consequence of Theorem \ref{thm:CT} is the following:

\begin{theorem} \label{thm:frob}
Let $f: X \to Y$ be a dominant morphism of smooth, proper and geometrically integral varieties over a number field $k$, with geometrically integral generic fibre. The set of places $v$ such that $Y(k_v)=f(X(k_v))$ is frobenian. \label{cor:frobenian}
\end{theorem}

Here we use the term ``frobenian'' in the sense of Serre \cite[\S 3.3]{Ser12} (see \S \ref{sec:frob}). Frobenian sets of places have a density, but being frobenian is much stronger than just having a density; for example, an infinite frobenian set has positive density.
It is also possible to prove Theorem \ref{thm:frob} using model-theoretic results and techniques such as quantifier elimination
\cite{Ax,Pas89}; our method avoids these and is completely algebro-geometric. However, we know of no model-theoretic proof of the finer Theorems \ref{maintheorem} and \ref{thm:CT}. (From a model-theoretic perspective, one may view Theorem \ref{thm:CT} as an explicit instance of quantifier elimination).

\subsection{} Some of the ingredients of our proof are already present in the work of Denef \cite{Denefbis, Denef}, e.g.~the use of the weak toroidalisation theorem \cite{AbramovichKaru,AbramovichDenefKaru}. We need more ingredients from logarithmic geometry, cf.~\S \ref{sec:loggeometry} -- essentially a few basic properties of log smooth morphisms and log blow-ups. The choice of a log smooth model for the morphism makes some of its arithmetic properties more transparent, and can be seen as a convenient way to come up with a Galois stratification, in the sense of Fried and Sacerdote \cite{FriedSacerdote}. On the other hand, we also use work of Serre \cite{Ser12} on frobenian functions, expanding upon what was done in \cite{LS}.

Let us finally give an overview of the structure of our paper. In \S \ref{sec:pseudosplit} we introduce the class of \emph{pseudo-split varieties} and discuss their elementary properties; this section also includes some examples involving torsors under coflasque tori. In \S \ref{sec:splittingdensities}, we introduce the ``$s$-invariants''. These allow us to prove in \S \ref{sec:nonsurjectivity} that our geometric conditions are \emph{necessary} for arithmetic surjectivity. To prove that this criterion is also \emph{sufficient}, we introduce the necessary logarithmic tools in \S \ref{sec:loggeometry}. We finish the proof of our main result in \S \ref{sec:surjectivity}. 

\subsection{Notation}  A variety over a field $k$ is a separated $k$-scheme of finite type.

Let $k$ be a  number field. We denote by 
$\Omega_{k,\f}$ the set of finite places of $k$. Given a place $v$ of $k$, we write $k_v$ for the completion of $k$ at $v$. If $v$ is non-archimedean, then we denote by $\OO_v$ the ring of integers of $k_v$, by $\FF_v$ its residue field, and by $\Norm(v) = \#\FF_v$ its norm. 

For a variety $X$ over a number field $k$, a \emph{model} of $X$ is a scheme $\mathcal{X}$ of finite type over $\OO_k$ together with a choice of isomorphism $X\cong \mathcal{X} \times_{\OO_k} k$. For a morphism $f : X \to Y$ of $k$-varieties, a \emph{model} of $f$ is a morphism $f:\mathcal{X} \to \mathcal{Y}$ of finite type over $\OO_k$, again denoted by $f$, of models of $X$ and $Y$, such that the induced map on generic fibres is identified with the original morphism $X \to Y$ via the isomorphisms $\mathcal{X} \times_{\OO_k} k \cong X$ and $\mathcal{Y} \times_{\OO_k} k \cong Y$.

\subsection*{Acknowledgements} We thank Dan Abramovich, Dami\'an Gvirtz, Ofer Gabber, Martin Ulirsch and Paul Ziegler for helpful discussions on technical geometric issues, and Raf Cluckers, Jan Denef and Jamshid Derakhshan for discussions on model-theoretic methods. We thank Yongqi Liang for contributing Remark \ref{remark:Liang}. We are grateful to the referee for a very careful reading, and for helpful suggestions for improvement. The first-named author was  supported by EPSRC grant EP/R021422/1. The second-named author was partially supported by EPSRC grant EP/M020266/1; the work on this paper was started when he was visiting the Institute for Advanced Study
in Princeton, where he was supported by the Charles Simonyi Endowment, and finished when he was visiting
the Max Planck Institut f\"ur Mathematik in Bonn. The third named author acknowledges the support of FWO Vlaanderen (postdoctoral fellowship) and NWO  (Veni grant), and is grateful to the Max Planck Institut f\"ur Mathematik in Bonn for its hospitality.

\section{Pseudo-split varieties} \label{sec:pseudosplit}

In this section, we collect some observations on the characterisation and properties of pseudo-split algebras and pseudo-split varieties. 

\subsection{Pseudo-split algebras}  \label{sec:pseudosplit_basic}

Let $k$ be a field, with algebraic closure $\overline{k}$, and let $A$ be a finite \'etale $k$-algebra; then $A = \prod_{i = 1}^n k_i$ for some finite separable field extensions $k_i/k$. Write $d_i=[k_i:k]$. Let $K_i/k$ be the Galois closure of $k_i/k$ in $\bar k$. Let $K$ be the compositum of $K_1,\ldots,K_n$, i.e.~the smallest subfield of $\bar k$ containing these fields.

Write $G=\Gal(K/k)$, $G_i=\Gal(K_i/k)$ and $H_i=\Gal(K_i/k_i).$ The index of $H_i$ in $G_i$ is equal to $d_i$.
By the normal basis theorem the $G$-module $A\otimes_k K$ is identified with $\bigoplus_{i=1}^n K[G_i/H_i]$,
where $G$ acts on $G_i/H_i$ through the natural homomorphism
$G\to G_i$. We have $(\Spec A)(\overline{k})=(\Spec A)(K)=\coprod_{i=1}^n G_i/H_i$.

\begin{definition}  \label{def:splitalgebra}
The following conditions are equivalent:
\begin{enumerate}
\item[(1)] $A=k\oplus A'$ for some $k$-algebra $A'$,
\item[(2)] the natural morphism $\Spec A\to \Spec(k)$ has a section,
\item[(3)] at least one point of $(\Spec A)(\bar k)$ is fixed by $G$.
\end{enumerate}
If these conditions are satisfied, we say that $A$ is \emph{split}. For example, a separable polynomial $p(X)\in k[X]$ has a root in $k$ if and only if
$k[X]/(p(X))$ is a split $k$-algebra.
\end{definition}

\begin{definition} \label{def:psalgebra}
The following conditions are equivalent:
\begin{enumerate}
\item[(1)] $A \otimes_k F$ is split for each field $F$ such that $k \subseteq F \subseteq K$ and $K/F$ is cyclic,
\item[(2)] each element of $G$ fixes at least one point of $(\Spec A)(\overline{k})$.
\end{enumerate}
If these conditions are satisfied, we say that $A$ is \emph{pseudo-split}.
\end{definition}

Condition (2) in the above definition can be rephrased by saying that
$G$ is the union of the stabilisers of points of $\coprod_{i=1}^n G_i/H_i$.
Let $\widetilde H_i\subset G$ be the inverse image of $H_i$ under the surjective 
homomorphism $G\to G_i$. Then $G_i/H_i=G/\widetilde H_i$ with its natural
$G$-action. Thus condition (2) is equivalent to the equality
\begin{equation}
G=\bigcup  g \widetilde H_i g^{-1}, \label{cov}
\end{equation}
where the union is taken over all $g\in G$ and $i=1,\ldots, n$. 

\begin{remark}
Pseudo-split algebras naturally arise in the study of the Hasse principle for finite schemes over a number field (see e.g.~\cite[Lemma~2.2]{JL15}
and \cite[Proposition~1]{Son08}).
\end{remark}

\begin{remark} 
By Jordan's theorem \cite[Theorem~4]{Ser03}, any transitive subgroup of a permutation group
contains a permutation without fixed points. Hence if $A$ is a pseudo-split
$k$-algebra such that $\Spec A$ is connected, then $A=k$. 
\end{remark}

\begin{remark} \label{basechange} One immediately sees that if a $k$-algebra $A$ is pseudo-split, then for any field extension
$L/k$ the $L$-algebra $A\otimes_k L$ is pseudo-split. \end{remark}

\begin{remark}
	If $G$ is cyclic, then any pseudo-split $k$-algebra is in fact split.
\end{remark}

The following gives a description of pseudo-split algebras over number fields.
The proof is an exercise using the Chebotarev density theorem; we omit it as the result will not
be used in the sequel (it is also a special case of Proposition \ref{prop:PS}).

\begin{proposition} Let $k$ be a number field. Then a finite \'etale $k$-algebra is pseudo-split if and only if it is split over almost
all completions of this field. In particular, a separable polynomial $p(x)\in k[x]$ has a root in almost all completions
of $k$ if and only if $k[x]/(p(x))$ is a pseudo-split $k$-algebra. \end{proposition}

See Lemma \ref{lem:is_pseudo-split} for a variant of this for finitely generated fields over $\QQ$.
Returning to a general field $k$ one can classify
all pseudo-split non-split $k$-algebras as follows.

\begin{proposition} Let $K/k$ be a Galois extension. 
Let $E_1,\ldots,E_n$ be subgroups of $G=\Gal(K/k)$
such that $G$ is the union of $gE_ig^{-1}$ for all $g\in G$
and $i=1,\ldots, n$. Then 
$\bigoplus_{i=1}^n K^{E_i}$ is a pseudo-split $k$-algebra,
and all pseudo-split $k$-algebras $A$ such that $A\otimes_k K$ is isomorphic
to $K^{\dim A}$ are obtained in this way. Under this bijection, the non-split $k$-algebras are those for which $E_1,\ldots,E_n$ are proper subgroups of $G$. \end{proposition}

\begin{proof} Write $A=\bigoplus_{i=1}^n K^{E_i}$. Then $(\Spec A)(\overline{k})$
is the set $\coprod_{i=1}^n G/E_i$. The subgroups $gE_ig^{-1}\subset G$
are precisely the $G$-stabilisers of points of this set. This shows that
$A$ is pseudo-split. Conversely, each pseudo-split $k$-algebra
$A$ can be written  in our previous notation as the direct sum of 
$k_i=K^{\widetilde H_i}$, where the subgroups $\widetilde H_i$ satisfy
(\ref{cov}). It is clear that $A$ is split if and only if $E_i=G$ for some $i$. \end{proof}

Any non-cyclic Galois extension $K/k$
gives rise to at least one pseudo-split non-split $k$-algebra: take
$E_1,\ldots,E_n$ to be all cyclic subgroups of $\Gal(K/k)$.

\begin{example} Let $K/k$ be a Galois extension such that $G = \Gal(K/k) \cong D_{n}$ is the dihedral group of degree $n$, where $n$ is odd. Then $D_{n} = \ZZ/n \rtimes \ZZ/2$, and $D_{n}$ is the union of $E_1 = \ZZ/n$ and the conjugates of $E_2 = \ZZ/2$. Hence we obtain a pseudo-split non-split algebra $A = k_1 \oplus k_2$, where $[k_1 : k] = 2$ and $[k_2 : k] = n$. For $n = 3$, we obtain a $5$-dimensional pseudo-split non-split $k$-algebra; this is the smallest possible dimension of such an algebra. If $k=\QQ$ and $K=\QQ(\sqrt{-3},\sqrt[3]{2})$, then $(X^2+3)(X^3-2)$ is solvable in all completions of $\QQ$ except $\QQ_2$ and $\QQ_3$. \end{example}

\subsection{Pseudo-split varieties} \label{sec:psv}
Let $k$ be a field (not necessarily perfect). 
For  a $k$-variety $X$, we denote by $X_{\mathrm{sm}}$   the maximal open subscheme of $X$ which is smooth over $k$. We let $X_{\mathrm{sm}, 1},\ldots, X_{\mathrm{sm}, n}$ be the irreducible components of $X_{\mathrm{sm}}$. Let $k_i$ be the algebraic closure of $k$ in the function field $k(X_{\mathrm{sm},i})$, for $i=1,\ldots,n$. 
Consider the finite $k$-algebra \begin{equation} \label{A_X} A_X=\bigoplus_{i=1}^n k_i.\end{equation} We call $\Spec A_X$ the {\em scheme of irreducible components of geometric multiplicity~$1$} of $X$. The map $X_{\mathrm{sm}} \to \Spec k$ factors as $X_{\mathrm{sm}}\to \Spec A_X \to \Spec k$, where $X_{\mathrm{sm}}\to \Spec A_X$ has geometrically integral fibres. Moreover, $\Spec A_X$ is smooth over $k$ by 
\cite[Lemma~34.11.5, Tag 05B5]{Stacksproject}, 
thus $\Spec A_X$ is finite and \'etale over $k$.

\begin{definition} \label{pseudosplitdefinition2} 
We say that $X$ is \emph{split} if the finite \'etale $k$-algebra $A_X$ is split in the sense of Definition~\ref{def:splitalgebra}. Similarly, we say that $X$ is \emph{pseudo-split} if the finite \'etale $k$-algebra $A_X$ is pseudo-split in the sense of Definition \ref{def:psalgebra}.
\end{definition}

If $X_{\mathrm{sm}} = \emptyset$, then \eqref{A_X} is the empty direct sum, hence $A_X$ is the zero ring and $\Spec A_X$ is the empty scheme. In this case it is follows easily from the definitions that $X$ is both non-split and non-pseudo-split.

The definitions in Definition \ref{pseudosplitdefinition2} are easily checked to be equivalent to Definitions \ref{def:split} and \ref{pseudosplitdefinition1} over a perfect field.
In this case, in the notation of \cite[\S 3.2]{LS}, being pseudo-split is equivalent to having $\delta(X)=1$. In the case where $k$ is a number field, we obtain the following characterisation.

\begin{proposition} \label{prop:PS}
Let $k$ be a number field. A $k$-variety is pseudo-split if and only if it has a smooth $k_v$-point for almost all completions $k_v$ of $k$.
\end{proposition}

\begin{proof} This follows immediately from \cite[Lemma~3.9]{LS}. \end{proof}

Another immediate consequence of the results in \cite[\S 3.2]{LS} is the following.

\begin{proposition} \label{birationality} Let $R$ be a discrete valuation ring with perfect residue field $k$. Let $\mathcal{X}_1$ and $\mathcal{X}_2$ be regular schemes which are proper over $R$, and whose generic fibres are birational. Then the special fibre of $\mathcal{X}_1$ is pseudo-split if and only if the special fibre of $\mathcal{X}_2$ is pseudo-split.
\end{proposition}

\begin{proof} This is a special case of \cite[Lemma 3.11]{LS}. \end{proof}

\subsection{Pseudo-split algebras and coflasque tori}

We make some observations in the style of Colliot-Th\'el\`ene's paper \cite{CTtores}. Given any finite \'etale $k$-algebra $A=\bigoplus_{i=1}^n k_i$, one defines the associated multinorm $k$-torus $\Res_{A/k}^1\Gm$ by the exact sequence
\begin{equation}
1 \to \Res_{A/k}^1\Gm \to \prod_{i = 1}^n \Res_{k_i/k} \GG_{m} \to \GG_{m} \to 1,
\label{cof}
\end{equation}
where the third map is induced by the norms from $k_i$ to $k$. In \cite{CTtores}, Colliot-Th\'el\`ene studied the special case $A = k(\sqrt{a}) \oplus k(\sqrt{b}) \oplus k(\sqrt{ab})$, where $a,b \in k^*$ are such that none of $a$, $b$ and $ab$ is a square and $\mathrm{char}(k) \neq 2$; this algebra is clearly pseudo-split.

\begin{proposition} If $A$ is a pseudo-split $k$-algebra, then $\smash{\Res_{A/k}^1\Gm}$ is a coflasque $k$-torus. If $k$ is a local field, then $\HH^1(k,\smash{\Res_{A/k}^1\Gm})=0$. \label{coflasque} \end{proposition}

We refer to \cite[\S1]{CTSansuc} for the definition of (and some background on) coflasque tori.

\begin{proof}
Let $M$ be the module of characters of $\smash{\Res_{A/k}^1\Gm}$ and let $G$ be as in \S\ref{sec:pseudosplit_basic}.  To prove that $\smash{\Res_{A/k}^1\Gm}$ is coflasque we need to show that $\HH^1(H,M)=0$ for each
subgroup $H\subset G$. By Remark \ref{basechange}, we can assume without loss of generality that $H = G$.

Consider the exact sequence of Galois-modules 
dual to (\ref{cof})
$$0\to\ZZ\to \bigoplus_{i=1}^n \ZZ[G/\widetilde H_i]\to M \to 0,$$
where we use the notation of \eqref{cov}.
The Galois action factors through the action of $G$, and $\HH^1(G,\mathbf{Z}[G/\widetilde H_i])=0$. To prove that
$\HH^1(G,M)=0$, we must therefore show that any element of the group $\mathrm{Hom}(G,\QQ/\ZZ)=\HH^2(G,\ZZ)$
which vanishes when restricted to each $\widetilde H_i$, is zero. This follows from \eqref{cov}, proving the first statement. The second statement is a general property of coflasque tori, as over the local field $k$ the finite abelian groups $\HH^1(k,\Res_{A/k}^1\Gm)$ and $\HH^1(k,M)=\HH^1(G,M)$ are dual to each other by Tate--Nakayama duality \cite[Corollary~I.2.4, Thm.~I.2.13]{Mil06}.
\end{proof}

Let $k$ be a number field and let $A=\oplus_{i=1}^n k_i$ be a pseudo-split $k$-algebra such that the extensions $k_i/k$ satisfy $\gcd([k_1:k],\ldots,[k_n:k]) \neq 1$. The family of torsors for $\smash{\Res_{A/k}^1\Gm}$   
\begin{equation} 
\prod_{i=1}^n N_{k_i/k}(x_i)= t\not=0 \label{fib2} 
\end{equation} 
can be compactified to a smooth, proper, geometrically integral variety with a morphism $\pi:X\to\PP^1_k$ extending the projection to the coordinate $t$. 

The map $X(k_v) \to \PP^1(k_v)$ is surjective for \emph{all} places $v$ of $k$. For smooth fibres this follows from Proposition \ref{coflasque}. However since $\pi$ is proper, the image $\pi(X(k_v))$ is closed, hence $\pi(X(k_v))=\PP^1(k_v)$ as claimed. The singular fibres are pseudo-split, but non-split. That they are pseudo-split is clear; that they are non-split follows from \cite[Lemma~5.4]{LS}. In particular, this surjectivity is not implied by Denef's result (Theorem \ref{deneftheorem}), but is implied by our Theorem~\ref{maintheorem} -- at least Theorem~\ref{maintheorem} gives surjectivity for all but finitely many $v$. It was this family of examples which in fact originally motivated Definition \ref{pseudosplitdefinition1}.

\section{Splitting densities} \label{sec:splittingdensities}
Let $k$ be a number field.
In this section we will introduce the ``$s$-invariants''; these are certain explicit frobenian functions which,
for a morphism of $k$-varieties $f:X \to Y$, 
measure the ``density'' of the split fibres of $X_{k_v} \to Y_{k_v}$ as $v$ varies.

\subsection{Frobenian functions} \label{sec:frob}
We first recall some of the theory of frobenian functions, following Serre's treatment in \cite[\S 3.3]{Ser12}. Recall that a function $\varphi:\Gamma \to \CC$ on a group $\Gamma$ is called a 
\emph{class function} if it is constant on each conjugacy class.
We denote by $\Omega_{k,\f}$ the set of finite places of $k$.

\begin{definition} \label{def:frob}
	A \emph{frobenian function} is a map $s: \Omega_{k,\f} \to \CC$ satisfying the following properties. There exist a finite Galois extension $K/k$ with Galois group $\Gamma$,
	a finite set of places $S \subset \Omega_{k,\f}$
	and a class function $\varphi: \Gamma \to \CC$ such that:
	\begin{enumerate}
		\item[(1)] $K/k$ is unramified outside of $S$;
		\item[(2)] $s(v) = \varphi(\Frob_v)$ for all $v \notin S$.
	\end{enumerate}
A subset of $\Omega_{k,\f}$ is called \emph{frobenian} if its indicator function is frobenian.
\end{definition}

Given $v \in \Omega_{k,\f}$ and a place $w \in \Omega_{K,\f}$ above $v$, we denote by $\Frob_{w/v} \in \Gamma$ the associated Frobenius element. In Definition \ref{def:frob}, we adopt a common abuse of notation (see \cite[\S3.2.1]{Ser12}), and denote by $\Frob_v \in \Gamma$ the choice of such an element for some $w$.
Note that $\varphi(\Frob_v)$ is well-defined as $\varphi$ is a class function.

\begin{example} 
Let $E/k$ be a finite extension of number fields. Then the set of all prime ideals of $\OO_k$ which split completely in $E$ is frobenian: in  Definition \ref{def:frob}, one takes $K$ to be the Galois closure of $E/k$ with Galois group $\Gamma$, $S$ the set of primes which ramify in $K/k$ and $\varphi: \Gamma \to \CC$ the indicator of the identity element of $\Gamma$. 
\end{example}
For any function $s: \Omega_{k,\f} \to \CC$, we define its \emph{density} to be
$$\dens(s) = \lim_{B \to \infty} \frac{\sum_{v \in \Omega_{k,\f}, \,\Norm(v) \leq B}s(v)}{B/\log B}\,,$$
if the limit exists. The density of a subset of $\Omega_{k,\f}$ is
defined to be the density of its indicator function. If $s$ is frobenian with associated class function $\varphi: \Gamma \to \CC$,
then we define its \emph{mean} to be
$$m(s) = \frac{1}{|\Gamma|} \sum_{\gamma \in \Gamma}\varphi(\gamma)$$ (this does not depend on the choice of $\varphi$).
A simple application of the Chebotarev density theorem (see \cite[\S 3.3.3.5]{Ser12})
shows that in this case $\dens(s)$ exists and
\begin{equation} \label{eqn:Cheb}
	m(s) = \dens(s).
\end{equation}
In particular, a frobenian set has positive density if and only if it is infinite.

\subsection{$s$-invariants} \label{sec:s-invariants}
We now define our $s$-invariants and study their properties.

\subsubsection{Set-up} \label{sec:s-invariants-set-up}

Let $k$ be a number field, let $K$ be a finitely generated field extension of $k$ and let $I$ be a non-empty finite \'{e}tale $K$-scheme. We associate to this situation some group
theoretic data as follows.

The $K$-algebra $K(I)$ is finite  \'{e}tale over $K$. Let $L$ be a finite Galois extension of $K$
such that $K(I)\otimes_KL\cong L^d$ for some $d \in \NN$. The Galois group $G=\Gal(L/K)$ acts on $I(L)$ 
in a natural way. 
We let $k_L$ (resp.~$k_K$) be the algebraic closure of $k$ in $L$ (resp.~$K$).

If $k_L/k$ is not Galois then we change $L$ as follows: let $M$ be a Galois closure of $k_L/k$ and let $L_M :=L \otimes_{k_L} M$. Note that $L_M$ is a field as $k_L$ is algebraically closed in $L$. Moreover $L_M$ is clearly still Galois over $K$ and the algebraic closure of $k_L$ in $L_M$ is $M$. In conclusion, replacing $L$ by $L_M$ if necessary, we may assume that $k_L/k$ is Galois.

Let $N$ be the normal subgroup of $G$ which acts trivially on $k_L$. Define $\Gamma = \Gal(k_L/k_K)$ and $\Lambda=\Gal(k_L/k)$. Note that $G/N = \Gamma \subset \Lambda$. We summarise this set-up with the following commutative diagram of field extensions and Galois groups.
\begin{equation} \label{diag:Galois_groups}
\begin{split}
\xymatrix{
k \ar@{-}[r] \ar@{=}[d] &	K \ar@{-}[r]^G \ar@{-}[d] & L \ar@{-}[d] \\
k \ar@{-}[r] \ar@/_1.5pc/@{-}[rr]^\Lambda & k_K \ar@{-}[r]^{\Gamma=G/N} & k_L
}
\end{split}
\end{equation}

\begin{definition}\label{def:s1}
	With the above set-up, let $v \in \Omega_{k,\f}$.
	\begin{itemize}
		\item	If $v$ ramifies in $k_L$ or there is no place $w$ of $k_K$ of degree $1$ 
		over $v$, then $s_I(v) := 1$. 
		\item Otherwise, we set
		\begin{equation} \label{def:s}
		\begin{split}
		s_I(v):=\frac{\displaystyle{\sum}_{\substack{w \in \Omega_{k_K} \\ \Norm w = \Norm v \\ w \mid v}} 
			\displaystyle{\frac{1}{|N|}}
			\#\left\{g \in G:
			\begin{array}{l}
			 g \bmod N = \Frob_w \text{ and } \\
			g \text{ acts with a fixed point on } I(L)
			\end{array}
			\right\}}
			{\#\{w \in \Omega_{k_K} : w \mid v, \Norm w = \Norm v\}}.
		\end{split}
		\end{equation}
	\end{itemize}
\end{definition}
\noindent Note that as
$$\#\left\{g \in G: \begin{array}{l} g \bmod N = \Frob_w \text{ and } \\	
g \text{ acts with a fixed point on } I(L)	\end{array}	\right\} \leq |N|,$$
we see that Definition \ref{def:s1} yields a well-defined function $s_I: \Omega_{k,\f} \to  [0,1]\cap\QQ$.

\subsubsection{Basic properties}
Let us first give a purely group-theoretic formula for the invariant which we introduced in Definition \ref{def:s1}.
For a subset $Z \subset H$ of a group $H$, we denote by $C_H(Z)$ the smallest subset of $H$
which is stable under conjugacy and contains $Z$, and  by $\Cl_H(Z)$ the set of conjugacy
classes in $C_H(Z)$.

\begin{lemma} \label{lem:is_frob}
	If $v$ ramifies in $k_L$ or $C_\Lambda(\Frob_v) \cap \Gamma= \emptyset$, then $s_I(v) = 1$. 
	Otherwise
	\begin{equation} \label{eqn:s_frob}
	s_I(v)	=\frac{\displaystyle{\sum}_{C } \displaystyle{\frac{1}{|C|\cdot |N|}}
		\#\left\{g \in G:
		\begin{array}{l}
		g \bmod N \in C \text{ and } \\
		g \text{ acts with a fixed point on } I(L)
		\end{array}\right\}}
		{\# \Cl_\Gamma(C_\Lambda(\Frob_v) \cap \Gamma)},
	\end{equation}
	where the sum is over $C \in \Cl_\Gamma(C_\Lambda(\Frob_v) \cap \Gamma)$.
\end{lemma}
\begin{proof}
	We claim that
	\begin{equation} \label{eqn:Frob_v_Frob_w}
		C_\Lambda(\Frob_v) \cap \Gamma = 
		\bigsqcup_{\substack{w \in \Omega_{k_K} \\ \Norm w = \Norm v \\ w \mid v }} C_\Gamma(\Frob_w).
	\end{equation}
Indeed, let $u$ be a finite place of $k_L$ over $v$ and let $w$ be its restriction to $k_K$. Then $$\Frob_{u/w} = \Frob_{u/v}^{[\FF_w:\FF_v]} \in \Gamma.$$
It follows that if $[\FF_w:\FF_v] = 1$, then $\Frob_{u/w} \in C_\Lambda(\Frob_v)$, so the left hand side of \eqref{eqn:Frob_v_Frob_w}
	contains the right hand side. Conversely, if $\Frob_{u/v} \in \Gamma$,
	then $\Frob_{u/v}$ leaves $k_K$ invariant, hence fixes $w$. Therefore $\Norm w = \Norm v$ and $\Frob_{u/w} = \Frob_{u/v}$, whence \eqref{eqn:Frob_v_Frob_w}.
	
	Using \eqref{eqn:Frob_v_Frob_w} we find that
	\begin{equation} \label{eqn:w_conjugacy}
		\#\{w \in \Omega_{k_K} : w \mid v, \Norm w = \Norm v\} = \#\Cl_\Gamma(C_\Lambda(\Frob_v) \cap \Gamma).
	\end{equation}
	This shows that $s_I(v) = 1$ if $C_\Lambda(\Frob_v) \cap \Gamma= \emptyset$ or if $v$ ramifies in $k_L$,
	by definition. So assume that we are not in these cases.
	By \eqref{eqn:w_conjugacy} we see that the denominators in \eqref{def:s} and \eqref{eqn:s_frob}
	agree.	As for the numerators, using \eqref{eqn:Frob_v_Frob_w} we obtain
	\begin{align*}
		&\sum_{\substack{w \in \Omega_{k_K} \\ \Norm w = \Norm v \\ w \mid v}} 
			\frac{1}{|N|}
			\#\left\{g \in G:
			\begin{array}{l}
			 g \bmod N = \Frob_w \text{ and } \\
			g \text{ acts with a fixed point on } I(L)
			\end{array}
			\right\} \\
		& = \sum_{C \in \Cl_\Gamma(C_\Lambda(\Frob_v) \cap \Gamma) }\sum_{\gamma \in C}  \frac{1}{|C| \cdot |N|}
			\#\left\{g \in G:
			\begin{array}{l}
			g \bmod N = \gamma \text{ and } \\
			g \text{ acts with a fixed point on } I(L)
			\end{array}\right\} \\
		& = \sum_{C \in \Cl_\Gamma(C_\Lambda(\Frob_v) \cap \Gamma) } \frac{1}{|C| \cdot |N|}
			\#\left\{g \in G:
			\begin{array}{l}
			g \bmod N \in C \text{ and } \\
			g \text{ acts with a fixed point on } I(L)
			\end{array}\right\}
	\end{align*}
	as required.
\end{proof}

\begin{corollary} \label{cor:is_frob}
	The function 
	$$s_I: \Omega_{k,\f} \to [0,1] \cap \QQ, \quad v \mapsto s_I(v),$$
	is frobenian.
\end{corollary}
\begin{proof}
	This follows from Lemma \ref{lem:is_frob}, which shows that 
	$s_I(v)$ only depends on the conjugacy class of $\Frob_v \in \Lambda$
	for all $v$ which are unramified in $k_L$.
\end{proof}

The invariant $s_I(v)$ simplifies in special cases.

\begin{example} \label{ex:s}
	Let $v$ be unramified in $k_L$.
	\begin{enumerate}
	\item
	Assume that $k = k_K$, i.e.~$K$ is geometrically irreducible. Then 
		\begin{align*}
		s_I(v)&=\frac{\#\left\{g \in G:
		\begin{array}{l}
		 g \bmod N = \Frob_v\text{ and }\\
		g \text{ acts with a fixed point on } I(L)
		\end{array}
		\right\}}{|N|}.
	\end{align*}

	\item
	Assume that $N=0$, e.g.~$K$ is a number field. If there is no place of $k_K$ of degree $1$ over $v$
	then $s_I(v) = 1$. Otherwise
	\begin{align*}
		s_I(v)&=\frac{\#\left\{w \in \Omega_{k_K} :
		\begin{array}{l}
		 w \mid v, \Norm w = \Norm v, \text{ and }\\
		\Frob_w \text{ acts with a fixed point on } I(L)
		\end{array}
		\right\}}
		{\#\{w \in \Omega_{k_K} : w \mid v, \Norm w = \Norm v\}}.
	\end{align*}
	\end{enumerate}
\end{example}

In special cases, one can relate the $s$-invariants to the $\delta$-invariants from \cite{LS}.

\begin{lemma}
	Assume that $k=k_K$. Then
	$$\dens(s_I) = \frac{\#\{ g \in G: g \emph{ acts with a fixed point on } I(L)\} }{\# G}.$$
	In particular $\dens(s_I) = \delta(I)$ in the notation of \cite[\S 3.2]{LS}.
\end{lemma}
\begin{proof}
	Since $k=k_K$, we have $\Lambda = \Gamma = G/N$. 
	Therefore Example \ref{ex:s} and the Chebotarev density theorem \eqref{eqn:Cheb} imply that 
	the density in question equals
	\begin{align*}
		&\frac{1}{|\Gamma|} \sum_{\gamma \in \Gamma}  
		\frac{\#\{g \in G: g \bmod N = \gamma \text{ and }
		g \text{ acts with a fixed point on } I(L)\}}{|N|} \\
		=  & \frac{1}{|\Gamma|\cdot|N|} \hspace{-20pt} 
		\sum_{\substack{g \in G \\ g \text{ acts with a fixed point on } I(L)}}
		\hspace{-40pt}
		\#\{ \gamma \in \Gamma : g \bmod N = \gamma \} \\
		= & \frac{\#\{ g \in G: g \text{ acts with a fixed point on } I(L)\} }{\# G},
	\end{align*}
	as required.
\end{proof}

\begin{example}
In general, one can have $\dens(s_I) \neq \delta(I)$ when $k \neq k_K$. Indeed, let $a,b\in k^*$ be such that $a,b,ab \notin k^{*2}$.
Take $K=k(\sqrt{b})$ and $I = \Spec k(\sqrt{a}, \sqrt{b})$ over $K$. In this case we have $G = \ZZ/2\ZZ$,
$\Gamma = G = \ZZ/2\ZZ$, $\Lambda= (\ZZ/2\ZZ)^2$ and $N=0$. Using Example \ref{ex:s}, one easily checks that
$3/4 = \dens(s_I) \neq \delta(I) = 1/2.$
\end{example}

\subsubsection{Determining when $s_I(v) = 1$}

Of particular interest to us will be the set of places with $s_I(v) = 1$. Here 
we have the following criterion.

\begin{lemma} \label{lem:s<1}
	Let $v$ be a place of $k$ which is unramified in $k_L$.
	Then $s_I(v) < 1$ if and only if there exists some $g \in G$ such that
	\begin{enumerate}
		\item[(1)] $g$ does not act with a fixed point on $I(L)$,
	\end{enumerate}
	and such that $g$ satisfies one of the following equivalent conditions.
	\begin{enumerate}
		\item[(2)] There is a place $w$ of $k_K$ of degree $1$ over $v$ such that $g \bmod N = \Frob_w$.
		\item[(3)] $g \bmod N \in C_\Lambda(\Frob_v) \cap \Gamma$.
	\end{enumerate}
\end{lemma}
\begin{proof}
	This follows immediately from Definition \ref{def:s1} and Lemma \ref{lem:is_frob}
\end{proof}

\begin{lemma} \label{lem:s=1_density}
	The set 
	\begin{equation} \label{def:s=1}
		\{v \in \Omega_{k,\f}: s_I(v) = 1\}
	\end{equation}
	is frobenian; its density is equal to
 	$$\frac{1}{|\Lambda|}	\#\left\{\lambda \in \Lambda:\begin{array}{l}
 		\mbox{each }g \in G \mbox{ with } g \bmod N \in C_\Lambda(\lambda) \cap \Gamma \\
		\mbox{acts with a fixed	point on } I(L)
	\end{array}
 	\right\}. $$
\end{lemma}
\begin{proof}
	That \eqref{def:s=1} is frobenian follows immediately from Corollary \ref{cor:is_frob}. The density is easily
	calculated using the Chebotarev density theorem \eqref{eqn:Cheb} and Lemma \ref{lem:s<1}.
\end{proof}

We now relate the $s$-invariants to the notion of pseudo-splitness introduced in \S \ref{sec:pseudosplit}.

\begin{lemma} \label{lem:is_pseudo-split}
	With notation as in \S\ref{sec:s-invariants-set-up}, the finite \'etale $K$-scheme $I$ is pseudo-split over $K$
	if and only if $s_{I}(v) = 1$ for all but finitely many places $v$ of $k$.
\end{lemma}
\begin{proof}
	It follows easily from Lemma \ref{lem:s=1_density} that the set
	\eqref{def:s=1} has density $1$ if and only if every element of $G$ acts with a fixed point of $I(L)$,
	i.e.~if and only if $I$ is pseudo-split. As  \eqref{def:s=1} is frobenian, this completes
	the proof.
\end{proof}

\subsection{$s$-invariants in families}
\label{subsec:families}
We now consider $s$-invariants in families and give an application to splitting densities over finite fields. 
We begin with some formalities concerning irreducible components in families.

\subsubsection{Irreducible components} \label{sec:Irr}
For a morphism of schemes $X \to Y$ we denote by $\Irr_{X/Y}$ the \emph{functor of open irreducible components} of $X$ over $Y$, defined by Romagny in \cite[D\'efinition~2.1.1]{Rom11}. We call $\Irr^1_{X/Y}:=\Irr_{X_{\mathrm{sm}}/Y}$ the \emph{subfunctor of open irreducible components of $X$ over $Y$ of geometric multiplicity $1$}, where $X_{\mathrm{sm}}$ denotes  the maximal open subscheme of $X$ which is smooth over $Y$. This parametrises those components of the fibres of $f$  which are geometrically generically reduced.

\begin{lemma} \label{lem:Irr}
	Let $X \to Y$ be a morphism of schemes of finite presentation, with $Y$ irreducible. Then there exists a dense  open subset $U \subset Y$ such that the restrictions of $\Irr_{X/Y}$ and $\Irr^1_{X/Y}$
	to $U$ are representable by finite \'{e}tale $U$-schemes.
\end{lemma}
\begin{proof}
	It clearly suffices to prove the result for $I:=\Irr_{X/Y}$.
By \cite[Lemmes~2.1.2 et 2.1.3]{Rom11}, there exists a dense open $U\subset Y$ such that
	$I|_U$ is representable by a quasi-compact algebraic space \'etale over $U$.
	Moreover by \cite[Proposition~2.1.4]{Rom11}, the generic fibre of $I$ is 
	a finite affine scheme. In particular the generic fibre is separated.
	As being separated is a constructible property
	(for schemes this is \cite[Proposition~9.6.1]{EGAIV}; the case 
	of algebraic spaces follows from similar
	arguments to those given in \cite[\S A]{Rom11}), on shrinking $U$
	we may assume that $I|_U \to U$ is separated.
	Thus $I$ is a scheme by Knutson's criterion \cite[Corollary~II.6.17]{Knu71}.
	Finally, as the generic fibre of $I|_U \to U$ 
	is finite, we may again shrink $U$ further to assume that $I|_U \to U$ is finite
	\'etale, as required.
\end{proof} 

\subsubsection{Definition}
We can now define $s$-invariants in families.
Let $f: X\to Y$ be a morphism of schemes of finite type over a number field $k$.
Let $y \in Y$ be an arbitrary point. The residue field $\kappa(y)$ is a finitely
generated extension of $k$. By Lemma \ref{lem:Irr}, the functor  
$\Irr^1_{f^{-1}(y)/\kappa(y)}$ of irreducible components of multiplicity $1$ of the fibre 
$f^{-1}(y)$ is representable by a finite \'{e}tale $\kappa(y)$-scheme. 
If this scheme happens to be empty, define $s_{f,y}(v)=0$ for all $v \in \Omega_{k,\f}$. Otherwise,
we are in the set-up of \S \ref{sec:s-invariants-set-up},
with $K= \kappa(y)$ and $I=\Irr^1_{f^{-1}(y)/\kappa(y)}$. 
Hence, given $v \in \Omega_{k,\f}$, we define 
\begin{equation} \label{def:s2}
	s_{f,y}(v) := s_{\Irr^1_{f^{-1}(y)/ \kappa(y)}}(v)
\end{equation}
using the notation of Definition \ref{def:s1}.

\subsubsection{Splitting densities over finite fields}
Our next proposition is the main result of this section,
and shows how $s$-invariants arise ``in nature''.

\begin{proposition} \label{prop:s}
Let $f: X\to Y$ be a morphism of schemes of finite type over a number field $k$,
with $Y$ integral. Let $f: \mathcal{X} \to \mathcal{Y}$ be a model of $X\to Y$ over $\OO_{k}$.
Let $n = \dim Y$ and let $\eta$ be the generic point	of $Y$. 	Then 
$$\#\{y \in \mathcal{Y}(\FF_v): f^{-1}(y) \mbox{ is split}\} = s_{f,\eta}(v)\#\mathcal{Y}(\FF_v) + O((\Norm v)^{n-1/2}), \quad \text{as }\Norm v \to \infty,$$
where the implied constant depends on $f$ and the choice of the model.
\end{proposition}

\begin{proof}	
	The key analytic ingredient for this result is \cite[Prop.~9.15]{Ser12},
	which is a version of the Chebotarev density theorem for arithmetic schemes.
	
	Choose a finite set $S$ of finite places of $k$, which we will enlarge 
	throughout the proof.	Let $I:= \Irr^1_{X/Y}$ and $\mathcal{I}:=\Irr^1_{\mathcal{X}/\mathcal{Y}}$.	
	First, by Lemma \ref{lem:Irr}, the restriction of $\mathcal{I}$ 
	to an open dense subset of $\mathcal{Y}$ is representable by a finite \'{e}tale cover.
	However, by the Lang--Weil estimates \cite{LW54}, strict closed subsets contribute to the error term only. We may therefore
	replace $\mathcal{Y}$ by a dense open subset,
	if necessary, to assume that 
	$\pi:\mathcal{I} \to \mathcal{Y}$ is finite \'{e}tale.
	Similarly, we may assume that $\mathcal{Y}$ is normal and that $\mathcal{Y}_{\FF_v}$
	is normal for all $v \notin S$.

	Our next step is to obtain a version of the diagram \eqref{diag:Galois_groups}
	over $\OO_{k,S}$. Let $k_Y$ denote the algebraic closure of $k$ inside $\kappa(Y)$.
	As $Y$ is normal and integral, the ring $\OO_Y(Y)$ is an integrally closed domain.
	In particular $k_Y \subset \OO_Y(Y)$, hence the structure morphism
	$Y\to\Spec k$ factors through a morphism $Y \to \Spec k_Y$. 
	As in \S\ref{sec:s-invariants-set-up} we choose a common Galois closure for the 
	connected components of $I$ over $Y$
	(for the existence of the Galois closure of a connected finite \'etale cover,
	see \cite[Prop.~5.3.9]{Sza09}).
	This yields a finite \'etale Galois morphism $\mathcal{L} \to \mathcal{Y}$
	with $\mathcal{L}$ integral such that, enlarging $S$ if necessary and choosing
	$\mathcal{L}$ appropriately as in \S\ref{sec:s-invariants-set-up}, we obtain
	a commutative diagram 
	\begin{equation} \label{diag:Galois_models}
	\begin{split}
	\xymatrix{
	\Spec \OO_{k,S}  \ar@{=}[d]  &	\mathcal{Y}_{\OO_{k,S}} \ar[l] \ar@{<-}[r]^G \ar[d] &   \mathcal{L}_{\OO_{k,S}}  \ar[d] \\
	\Spec \OO_{k,S} \ar@/_1.5pc/@{<-}[rr]^\Lambda & \Spec \OO_{k_Y,S} \ar[l] \ar@{<-}[r]^{\Gamma = G/N} & \Spec \OO_{k_L,S} 
	}
	\end{split}
	\end{equation}
	which recovers the diagram \eqref{diag:Galois_groups} on the level
	of function fields. Here $L = \kappa(\mathcal{L})$ and $k_L$ is the algebraic closure
	of $k$ in $L$. We  denote by $\mathcal{Y}_{\OO_{k,S}}$ and $\mathcal{L}_{\OO_{k,S}}$
	the base change to $\OO_{k,S}$, and abuse notation by denoting $\OO_{k_Y,S}$ and $\OO_{k_L,S}$
	the localisations of  $\OO_{k_Y}$ and $\OO_{k_L}$ respectively at those places which lie above places of $S$. We also choose $S$ sufficiently large so that
	each morphism in the bottom row is finite \'{e}tale.
	
	Let $v$ be a finite place of $k$ not in $S$.
	We now give an asymptotic formula for $\#\mathcal{Y}(\FF_v)$.
	Enlarging the set $S$ if necessary, we may assume that all fibres of $\mathcal{Y}_{\OO_{k,S}} \to \mathrm{Spec}\, \OO_{k_Y,S}$ are geometrically integral. Therefore the irreducible components of 
	$\mathcal{Y}_{\FF_v}$ are in bijection with those places $w$ of $k_Y$ which divide $v$. 
We denote the corresponding component by $\mathcal{Y}_w$; this is geometrically 
	irreducible over $\FF_v$ if and only if $\Norm w = \Norm v$. The $\mathcal{Y}_w$ are disjoint as $\mathcal{Y}_{\FF_v}$
	is normal \cite[Tag 033M]{Stacksproject} and, again by normality, any non-geometrically-irreducible
	component has no $\FF_v$-point. Thus if there is no  $w$ with $\Norm w = \Norm v$, then 
	the proposition trivially holds. 
	So assume that there is a place $w$ of $k_Y$ with $\Norm w = \Norm v$. The Lang--Weil estimates now yield
	\begin{equation}
		\#\mathcal{Y}(\FF_v)  = 
		\sum_{\mathclap{\substack{w \in \Omega_{k_Y}  \\ \Norm w = \Norm v \\ w \mid v }}} \#\mathcal{Y}_w(\FF_w) 
		 = (\Norm v)^n  \sum_{\mathclap{\substack{w \in \Omega_{k_Y}  \\ \Norm w = \Norm v \\ w \mid v }}}1  
		 + O((\Norm v)^{n-1/2}).		\label{eqn:D}
	\end{equation}
	
	Next let $w$ be a place of $k_Y$ with
	$w \mid v$. For $y \in \mathcal{Y}_w(\FF_w)$, we let $\Frob_y \in G$
	denote the choice of some Frobenius element of $y$ (well-defined up to conjugacy; 
	see \cite[\S 9.3.1]{Ser12}). Note that $f^{-1}(y)$ is split
	if and only if $\Frob_y$ acts with a fixed point on $I(L)$. We therefore let
	$F: G \to \{0,1\}$ be the indicator function of those
	elements of $G$ which act with a fixed point on $I(L)$, which is a class function on $G$.
	We obtain
	
	$$\#\{y \in \mathcal{Y}_w(\FF_w): f^{-1}(y) \mbox{ is split}\} = \sum_{\mathclap{y \in \mathcal{Y}_w(\FF_w)}} F(\Frob_y).$$
	We are now in the set-up of \cite[Prop.~9.15]{Ser12}. Hence we may apply \emph{loc.~cit}.~to obtain
	$$\#\{y \in \mathcal{Y}_w(\FF_w): f^{-1}(y) \mbox{ is split}\} =  F^N(\Frob_w) (\Norm w)^n + O((\Norm w)^{n-1/2}),$$
	where for $\gamma \in \Gamma$ we define
	$$F^N(\gamma) = \frac{1}{|N|}\sum_{g \bmod N = \gamma}F(g)$$
	(cf.~\cite[\S 5.1.4]{Ser12}). 
	As in the proof of \eqref{eqn:D} we obtain
	\begin{align*}
		&\#\{y \in \mathcal{Y}(\FF_v): f^{-1}(y) \mbox{ is split}\} \\
		&= \frac{(\Norm v)^n}{|N|}
		\sum_{\mathclap{\substack{w \in \Omega_{k_Y}  \\ \Norm w = \Norm v \\ w \mid v }}} 
		\#\left\{g \in G:
		\begin{array}{l}
			g \bmod N = \Frob_w \text{ and } \\
			g \text{ acts with a fixed point on } I(L)
		\end{array} \right\}
		+ O((\Norm v)^{n-1/2}).
	\end{align*}
	Combining this with \eqref{eqn:D} and recalling Definition \ref{def:s1} completes the proof.
\end{proof}

\begin{corollary} \label{cor:non_sur}
	Let $f: X\to Y$ be a morphism of schemes of finite type over a number field $k$,
	with $Y$ integral. Let $f: \mathcal{X} \to \mathcal{Y}$ be a model of $X\to Y$ over $\OO_{k}$.
	
	There exists a finite set of places $S$ of $k$ such that for all $v \notin S$ with $s_{f,\eta}(v) < 1$,
	there exists $y \in \mathcal{Y}(\FF_v)$ such that $f^{-1}(y)$ is non-split.
\end{corollary}
\begin{proof}
	As $s_{f,\eta}(v) < 1$, Definition \ref{def:s1} implies that
	there is a place $w$ of $k_Y$ of degree $1$ over $v$; hence $Y_{k_v}$ is split.
	The Lang--Weil estimates \cite{LW54} show that $\mathcal{Y}(\FF_v) \neq \emptyset$
	for all $v \not \in S$, for a suitably large set of places $S$. Enlarging $S$ if necessary, Proposition~\ref{prop:s} implies that there exists $y \in \mathcal{Y}(\FF_v)$
	for which $f^{-1}(y)$ is non-split, as required.
\end{proof}

\begin{corollary} \label{cor:sur}
	Let $f: X\to Y$ be a morphism of schemes of finite type over a number field $k$,
	with $Y$ normal and integral. Assume that $\Irr^1_{X/Y}$ is representable by
	a finite \'etale scheme over $Y$. Let $f: \mathcal{X} \to \mathcal{Y}$ be a model of $X\to Y$
over $\OO_{k}$.
	
	There exists a finite set of places $S$ of $k$ such that for all $v \notin S$,
	the fibre over \emph{every} point $y \in \mathcal{Y}(\FF_v)$ is split
	if and only if $s_{f,\eta}(v) = 1$.
\end{corollary}
\begin{proof}
	By Lemma \ref{lem:Irr} there is a dense open subset $\mathcal{U} \subset \mathcal{Y}$
	such that $\Irr^1_{\mathcal{X}_\mathcal{U}/\mathcal{U}}$ is representable by a finite \'etale
	scheme over $\mathcal{U}$. The restriction $\Irr^1_{X/Y}$ of the functor $\Irr^1_{\mathcal{X}/\mathcal{Y}}$
to $Y$ is representable by a finite \'etale scheme over $Y$, by assumption.
	As $\Irr^1_{\mathcal{X}/\mathcal{Y}}$ is an \'etale sheaf
	\cite[Lemme~2.1.2]{Rom11} we may glue these representations together. Thus we may
assume that the generic fibre of $\mathcal{U}\to \Spec\OO_k$ is $Y$. Spreading out
we see that there is a finite set of places $S$ of $k$ such that $\Irr^1_{\mathcal{X}_{\OO_{k,S}}/\mathcal{Y}_{\OO_{k,S}}}$
	is representable by a finite \'etale scheme.
	
	To continue, we use some of the techniques from the proof of Proposition \ref{prop:s} and
	keep the notation of that proof.
	We will use the diagram \eqref{diag:Galois_models}; this is valid on enlarging $S$,
	because $\Irr^1_{X/Y} \to Y$ is finite \'{e}tale and $Y$ is normal. 
	
	If $s_{f,\eta}(v) < 1$ then, enlarging $S$,
	the result follows from Corollary \ref{cor:non_sur}.
	So let $v \notin S$,  assume that $s_{f,\eta}(v) = 1$ and let $y \in \mathcal{Y}(\FF_v)$.
	Let $w$ be the place of $k_Y$ lying below $y$ in \eqref{diag:Galois_models};
	this has degree $1$ over $v$ as $y \in \mathcal{Y}(\FF_v)$. 
	Let $l \in \mathcal{L}$ be a closed point lying above $y$ and
	let $u$ be the place of $k_L$ lying below $l$ in \eqref{diag:Galois_models}.
	Let $\FF_w,\FF_u,\FF_l$ and $\FF_y$ be the respective residue fields.
	From \eqref{diag:Galois_models} we obtain the tower of extensions of finite fields
	\begin{equation*} 
		 \FF_y = \FF_w \subset \FF_u \subset \FF_l.
	\end{equation*}
	The functoriality of Frobenius elements in extensions of 
	finite fields implies that
	\begin{equation} \label{eqn:mod_Frob}
		\Frob_{l/w} \, \bmod N = \Frob_{u/w},
	\end{equation}
	where $\Frob_{l/w} \in \Gal(\FF_l/\FF_w) = \Gal(\FF_l/\FF_y) \subset G$ and 
	$\Frob_{u/w} \in \Gal(\FF_u/\FF_w) \subset \Gamma$ denote the associated Frobenius elements.
	However, as $s_{f,\eta}(v) = 1$,
	Lemma \ref{lem:s<1} and \eqref{eqn:mod_Frob} imply that $\Frob_{l/w} = \Frob_{l/y} \in G$ acts with a fixed point on $I(L)$.
	Thus $f^{-1}(y)$ is split.
\end{proof}




\section{Non-surjectivity} \label{sec:nonsurjectivity}

Using the material developed in \S \ref{sec:splittingdensities}, we now prove one implication of Theorem \ref{maintheorem}, namely that our geometric conditions are \emph{necessary} for arithmetic surjectivity. We require the following criterion for non-existence of an $\OO_v$-point in a fibre.

\begin{proposition} \label{thm:sparsity} 
Let $k$ be a number field. Let $f: X \to Y$ be a dominant morphism of smooth and geometrically integral $k$-varieties
and let $f:\mathcal{X} \to \mathcal{Y}$ be a model over $\OO_k$. 
Let $\mathcal{T}$ be a reduced divisor in $\mathcal{Y}$ such that the restriction of $\mathcal{X} \to \mathcal{Y}$
to $\mathcal{Y}\setminus\mathcal{T}$ is smooth. 

There exist a finite set $S \subseteq \Omega_{k,\mathrm{f}}$ and a closed subset $\mathcal{Z}\subset\mathcal{T}_{\OO_{k,S}}$ containing the singular locus of $\mathcal{T}_{\OO_{k,S}}$, of codimension $2$ in $\mathcal{Y}_{\OO_{k,S}}$, such that for all finite places $v \notin S$ the following holds:

\begin{quote} Let $\mathcal{P} \in \mathcal{Y}(\OO_v)$ be such that the image of $\mathcal{P}: \mathrm{Spec}\,\OO_v \to \mathcal{Y}$ meets $\mathcal{T}_{\OO_{k,S}}$ transversally outside of $\mathcal{Z}$ and such that the fibre above $\mathcal{P} \bmod v \in \mathcal{T}(\FF_v)$ is non-split. Then 
$(\mathcal{X} \times_{\mathcal{Y}} \mathcal{P})(\OO_v)=\emptyset$. \end{quote}
\end{proposition}
\begin{proof}
For rational points this is proved in \cite[Thm.~2.8]{LS}. The adaptation to integral points is straightforward.
\end{proof}

Here is the main result of this section.

\begin{theorem} \label{thm:not_surjective}
Let $k$ be a number field. Let $f: X \to Y$ be a dominant morphism of smooth and geometrically integral $k$-varieties
and let $f: \mathcal{X} \to \mathcal{Y}$ be a model over $\OO_k$. Let $D \in Y^{(1)}$. Then there exists
a finite set $S$ of finite places of $k$ such that for all finite places $v \notin S$, the following holds: if $s_{f,D}(v) < 1$ then $\mathcal{X}(\OO_v) \to \mathcal{Y}(\OO_v)$
	is not surjective.
\end{theorem}
\begin{proof}
	Let $\mathcal{D}$ be the closure of $D$ inside $\mathcal{Y}$. Enlarge $S$ so that Proposition $\ref{thm:sparsity}$ may be applied, and let $\mathcal{T}$ and $\mathcal{Z}$ be as in the statement of that proposition. Enlarging $S$ further if necessary, we may assume that $\mathcal{Y}_{\OO_{k,S}}$ is smooth over $\OO_{k,S}$.

	Note that $s_{f,D}(v)$ only depends on $f^{-1}(D)$, i.e.~on the generic fibre of $\mathcal{D}$.
	In particular, enlarging $S$ further if necessary, we may apply Corollary~\ref{cor:non_sur} to the restriction of $f$ to 	$\mathcal{D}\setminus \mathcal{Z}$. Thus if $v \notin S$ and $s_{f,D}(v) < 1$, then there exists a point 
	$y \in \mathcal{D}(\FF_v) \setminus \mathcal{Z}(\FF_v)$ such that $f^{-1}(y)$ is non-split.
	
Therefore, by Proposition \ref{thm:sparsity}, to prove the result it suffices to show that
	we may lift $y$ to a point $\mathcal{P} \in \mathcal{Y}(\OO_v)$ 
	which meets $\mathcal{D}$ transversally at $y$. This is well-known; let us give a short geometric proof of this statement. Let $\psi: \mathcal{Y}' \to \mathcal{Y}$ be the blow-up of $\mathcal{Y}$ at $y$. The exceptional divisor $E=\mathbf{P}(T_{\mathcal{Y},y})$ is the projectivisation of the tangent space to $\mathcal{Y}$ at $y$ (since $\mathcal{Y}$ is smooth at $y$). The linear subvariety $\mathbf{P}(T_{\mathcal{D},y})$ of $E$ is strictly smaller (since $\mathcal{D}$ is smooth at $y$). Choose any $y' \in E(\mathbf{F}_v)$ which does not lie in this linear subvariety. By Hensel's lemma, $y'$ can be lifted to an 
$\OO_v$-point $\mathcal{P}' \in \mathcal{Y}'(\OO_v)$. The image $\mathcal{P} = \psi(\mathcal{P}') \in \mathcal{Y}(\OO_v)$ now satisfies the requirements.
\end{proof}

The following result (see \cite[Observation 2.2]{Denef}) follows from Greenberg's theorem \cite{Gre66}.

\begin{lemma} \label{lem:Denef_obs} Let $R$ be an excellent henselian discrete valuation ring with fraction field $K$. Let $f:X \to Y$ be a dominant morphism of integral schemes which are separated and of finite type over $R$, with $X \times_R K$ and $Y \times_R K$ smooth over $K$. Let $f':X' \to Y'$ be a modification of $f$, i.e.~a commutative diagram as in \eqref{def:modification}; here $\alpha_X:X' \to X$ and $\alpha_Y:Y'\to Y$ are proper birational morphisms over $R$, and $f':X' \to Y'$ is a dominant morphism of integral separated schemes over of finite type over $R$, with smooth generic fibre. Then $X(R) \to Y(R)$ is surjective if and only if $X'(R) \to Y'(R)$ is surjective.	
\end{lemma}

We now prove one implication of  Theorem \ref{maintheorem}.

\begin{corollary} \label{non-sur}
Let $f:X \to Y$ be a dominant morphism of smooth, geometrically integral varieties
over a number field $k$. Assume that there exists a modification $f': X' \to Y'$ of $f$, with $X'$ and $Y'$ smooth, and a point $D \in (Y')^{(1)}$ such that the fibre $(f')^{-1}(D)$ is not pseudo-split over $\kappa(D)$. Let $f: \mathcal{X} \to \mathcal{Y}$ be a model of $X\to Y$ over $\OO_k$. Then there exists a set of places $v$ of positive density 
such that the map $\mathcal{X}(\OO_v) \to \mathcal{Y}(\OO_v)$ is not surjective.
\end{corollary}
\begin{proof}
Ignoring finitely many places, by Lemma \ref{lem:Denef_obs} we may assume that $X=X'$, $Y=Y'$, $f=f'$.
By Lemmas \ref{lem:s=1_density} and \ref{lem:is_pseudo-split}, there exists a set of places $v$
of positive density such that $s_{f,D}(v) < 1$. The result then follows from Theorem \ref{thm:not_surjective}.
\end{proof}

\section{Logarithmic geometry} \label{sec:loggeometry}

\subsection{Preliminaries} The results and proofs in this section are written in the language of logarithmic geometry. Although it would be possible to use the older language of toroidal 
embeddings, the logarithmic framework is a convenient and flexible language, which makes arguments more conceptual and transparent. 
A good reference for basic terminology on log schemes is Kato's foundational paper \cite{Kato1}. Let us briefly recall the notions which are most essential for us. We work exclusively with \emph{Zariski log schemes}:

\begin{definition} A \emph{Zariski log scheme} is a pair $(X,\mathcal{M}_X)$, where $X$ is a scheme and $\mathcal{M}_X$ is a sheaf of monoids on $X$ (for the Zariski topology), equipped with a homomorphism of sheaves of monoids $\alpha_X: \mathcal{M}_X \to (\OO_X,\cdot)$ inducing an isomorphism $$\alpha_X^{-1}(\OO_X^*) \cong \OO_X^*.$$
\end{definition}

The fundamental example for the purpose of this paper is the situation where the \emph{log structure} $\alpha_X: \mathcal{M}_X \to \OO_X$ on $X$ is \emph{divisorial}, i.e.~induced by a divisor on $X$:

\begin{example} Let $X$ be any scheme and let $D$ be a divisor on $X$. Denote the corresponding open immersion by $j: X \setminus D \hookrightarrow X$. Then the monoid $$\mathcal{M}_X = j_* \OO^*_{X \setminus D} \cap \OO_X$$ (together with the natural inclusion $\alpha_X: \mathcal{M}_X \hookrightarrow \OO_X$) defines a log structure on $X$, the \emph{divisorial log structure induced by $D$}. This yields a Zariski log scheme $(X,\mathcal{M}_X)$, which we will sometimes denote by $(X,D)$ depending on the context. \label{example:divisorial}
\end{example}

Zariski log schemes form a category, with morphisms defined in the obvious way, cf.~\cite[\S 1.1]{Kato1}. We will almost always work with a full subcategory: the category of \emph{fine log schemes} \cite[Definition 2.9]{Kato1}, or the category of \emph{fs (``fine and saturated'') log schemes}.

\subsubsection{Log regular schemes and their fans} \label{subsubsection:fans}

Given an arbitrary monoid $P$, we write $P^*$ for the subgroup of invertible elements, $P^\sharp = P/P^*$ for the associated sharp monoid and $P^\mathrm{gp}$ for the group envelope of $P$.
(A monoid is \emph{sharp} if 0 is the only invertible element.)

\begin{definition} A \emph{log regular scheme} is a Zariski log scheme which is log regular. \end{definition}

Recall that a Zariski log scheme $X$ is said to be log regular if it is fs and the following property is satisfied for every $x \in X$: if $I(x,\mathcal{M}_X)$ is the ideal of $\OO_{X,x}$ generated by $\mathcal{M}_{X,x} \setminus \OO_{X,x}^*$, then $\OO_{X,x}/I(x,\mathcal{M}_X)$ is a regular local ring and $$\dim \OO_{X,x} = \dim \OO_{X,x}/I(x,\mathcal{M}_X) + \mathrm{rk}_{\mathbf{Z}}\,(\mathcal{M}_{X,x}^\sharp)^\mathrm{gp}.$$ 
Log regular schemes are normal and Cohen--Macaulay by \cite[Theorem 4.4]{Kato2}. Moreover, they  admit the following description by \cite[Theorem 11.6]{Kato2}.

\begin{remark} Let $(X,\mathcal{M}_X)$ be a log regular scheme. Let $U$ be the largest open subset of $X$ on which the log structure $\mathcal{M}_X$ is trivial, i.e.~coincides with the sheaf $\OO_X^*$ of invertible functions. Let $j: U \hookrightarrow X$ be the corresponding open immersion. Then $X \setminus U$ is a divisor and $\mathcal{M}_X =  j_*\OO_U^* \cap \OO_X$, as in Example \ref{example:divisorial}. \end{remark}

To a log regular scheme $(X,\mathcal{M}_X)$ one can associate a useful combinatorial object, its \emph{fan}. This notion has been introduced by Kato \cite[\S 9.1, \S 9.3]{Kato2}:

\begin{definition} \label{def:fan} A \emph{fan} is a locally monoidal space $(F,\mathcal{M}_F)$ which admits an open covering by affine monoidal spaces $\mathrm{Spec}\,P$, where $P$ is an fs monoid. A fan is called \emph{smooth} if it has an open covering by monoidal spaces of the form $\mathrm{Spec}\,\NN^r$; see \cite[Definition 4.11]{ACMUW}.
\end{definition} 

An affine fan $\mathrm{Spec}\,P$ is sometimes called a \emph{Kato cone}, cf.~\cite[\S 4]{ACMUW}. The definition of the sheaf of monoids $\mathcal{M}_{\mathrm{Spec}\,P}$
is motivated by the definition of an affine scheme. The stalks of
$\mathcal{M}_{\mathrm{Spec}\,P}$ are sharp, hence the stalks of $\mathcal{M}_F$ are sharp too.

To an arbitrary log regular scheme, Kato associates a fan as follows:

\begin{definition} Let $(X,\mathcal{M}_X)$ be a log regular scheme. Consider the monoidal space $F(X)$ with underlying set $\{x \in X : I(x,\mathcal{M}_X) = \mathfrak{m}_x\}$; here $I(x,\mathcal{M}_X)$ is defined as above, and $\mathfrak{m}_x$ denotes the maximal ideal of $\OO_{X,x}$. The topology on $F(X)$ is the subspace topology induced by the Zariski topology on $X$. The sheaf of monoids on $F(X)$ is the restriction to $F(X)$ of the sheaf $\mathcal{M}_X^\sharp$ of (sharp) monoids on $X$.
\end{definition}

Kato proves that $F(X)$ is indeed a fan \cite[Theorem 10.1]{Kato2}. If $X$ is quasi-compact, then the underlying set of $F(X)$ is finite. If $X$ is a regular scheme and if the log structure on $X$ is the divisorial log structure associated to a divisor on $X$ with strict normal crossings, then $F(X)$ is smooth (see \cite[\S 4]{ACMUW} and the references given there).

There exists a continuous, open morphism of monoidal spaces $\pi: (X,\mathcal{M}_X^\sharp) \to F(X)$ \cite[\S 10.2]{Kato2}. This allows one to stratify the scheme $X$ into locally closed subsets:

\begin{definition} Given a log regular scheme $(X,\mathcal{M}_X)$ and given $x \in F(X)$, denote $U(x) = \pi^{-1}(x)$, with $\pi: (X,\mathcal{M}_X) \to F(X)$ as above. When equipped with its reduced subscheme structure, this is a locally closed subscheme of $X$. This yields the so-called \emph{logarithmic stratification} $(U(x))_{x \in F(X)}$ of $X$ into locally closed subsets. \end{definition}

With notation as above, we denote by $\overline{U}(x)$ the Zariski closure of $U(x)$. Since $\pi$ is continuous and open, $\overline{U}(x)$ is the inverse image under $\pi$ of the closure of $x$ in $F(X)$. Any stratum $U(x)$ is regular and irreducible, but its closure $\overline{U}(x)$ may very well be singular.

The points of a Kato fan $F$ are in a natural bijection with the Kato subcones of
$(F,\mathcal{M}_F)$: every subcone has a unique closed point 
and can be recovered as the smallest open affine neighbourhood of this point,
see \cite[Lemma 4.6]{ACMUW}.

Let $F(\NN)=\Hom(\Spec\,\NN,F)$.
For a smooth fan $F$ we define the {\em height} of a morphism $g:\Spec\,\NN\to F$
as follows. The morphism $g$ sends the closed point $\NN_{>0}$ of $\Spec\,\NN$
to the closed point of a unique Kato subcone $\Spec\,\NN^r$, so 
$g$ factors through $\Spec\,\NN\to\Spec\,\NN^r$.
The dual map is a morphism of monoids $\NN^r\to\NN$;
we define the height $h_F(g)\in\NN$ to be the image of the sum of the canonical generators
of $\NN^r$. It is clear that for any $m\in\NN$ the set
$F(\NN)_{\leq m}=\{P\in F(\NN) : h_F(P)\leq m\}$ is finite.

\subsubsection{Log smooth morphisms and log blow-ups} \label{logblowups}

Log smooth morphisms can be defined in many different ways. We recall the definition which is usually called \emph{Kato's criterion}, stated and proven for \'etale log structures in \cite[Theorem 3.5]{Kato1}. The appropriate version for Zariski log structures seems to be folklore and can be found in \cite[Corollary 12.3.37]{GR}.

\begin{theorem} \label{theorem:Kato}
 Let $f: (X, \mathcal{M}_X) \to (Y,\mathcal{M}_Y)$ be a morphism of Zariski fs log schemes. Then $f$ is log smooth (resp. log \'etale) if and only if the following condition is satisfied. Given any point $x \in X$, an affine open neighbourhood $V = \mathrm{Spec}\,B$ of $f(x)$ in $Y$, and a chart $Q \to B$ for the log structure around $y$, there exist

\begin{itemize}
\item[$-$] an \emph{\'etale} affine neighbourhood $g: U = \mathrm{Spec}\,A \to X$ of $x$,
\item[$-$] a chart $P \to A$ for the log structure $g^* \mathcal{M}_X$ on $U$, and
\item[$-$] a homomorphism $\varphi: Q \to P$ yielding a chart for $(U,g^* \mathcal{M}_X) \to (V,\mathcal{M}_V),$
\end{itemize}
such that both of the following conditions are satisfied:
\begin{itemize}
\item[$-$] $\mathrm{ker}\,\varphi^\mathrm{gp}$ and the torsion of $\mathrm{coker}\,\varphi^\mathrm{gp}$ (resp. $\mathrm{ker}\,\varphi^\mathrm{gp}$ and $\mathrm{coker}\,\varphi^\mathrm{gp}$) are finite abelian groups, the orders of which are invertible on $U$;
\item[$-$] the induced morphism $U \to V \times_{\mathrm{Spec}\,\mathbf{Z}[Q]}\mathrm{Spec}\,\mathbf{Z}[P]$ is classically smooth.
\end{itemize}
\end{theorem}

In fact, an equivalent version of the criterion says that $U \to V \times_{\mathrm{Spec}\,\mathbf{Z}[Q]}\mathrm{Spec}\,\mathbf{Z}[P]$ can even be taken to be \'etale instead of smooth (see \cite[Theorem 3.5]{Kato1}).

An important basic fact is that any fs log scheme which is log smooth over a log regular scheme is again a log regular scheme (this is \cite[Theorem 8.2]{Kato2}). 

An essential class of log smooth and log \'etale morphisms is the class of log blow-ups. 
They are constructed in terms of subdivisions.
A \emph{subdivision} of a fan $F$ as in Definition \ref{def:fan} is a morphism of fans $\varphi: F' \to F$ satisfying the properties in \cite[Definition 9.6]{Kato2}. The following key example
will be used in the proof of Theorem \ref{maintheorem}.

\begin{example} \label{bary}
The \emph{barycentric subdivision} $B(F)$ of a fan $F$ is defined as follows.
The \emph{barycentre} of a Kato cone $\Spec\, P$ is a
canonical morphism $\Spec\,\NN\to \Spec P$ which, under the canonical isomorphism
$\Hom(\Spec\,\NN,\Spec\,P)=\Hom(P,\NN)$, corresponds to the map sending the
generator of each 1-dimensional face of $P$ to 1. One constructs $B(F)$
by performing the so-called star subdivision of each cone of $F$ along its barycentre,
in decreasing order of dimension. See \cite[Example 4.10 (ii)]{ACMUW} for details.
\end{example}

To a subdivision $\varphi: F' \to F(X)$ of the fan of a log regular scheme $(X,\mathcal{M}_X)$ 
one associates a log regular scheme $(X',\mathcal{M}_X')$ with $F(X')=F'$
and a morphism, called a \emph{log blow-up},
$$\mathrm{Bl}_\varphi: (X',\mathcal{M}_X') \to (X,\mathcal{M}_X)$$
such that the induced map $F(X') \to F(X)$ is exactly $\varphi$. 
The morphism $\mathrm{Bl}_\varphi$ is log \'etale and birational \cite[Proposition 10.3]{Kato2}. 
If $\varphi$ is proper \cite[Definition 9.7]{Kato2},
then the natural map $\varphi_*:F'(\mathbf{N}) \to F(X)(\mathbf{N})$ is bijective;
in this case $\mathrm{Bl}_\varphi$ is a proper morphism \cite[Proposition 9.11]{Kato2}.

Log blow-ups can be used to resolve  singularities of log regular schemes, as explained by Kato in \cite[\S 10.4]{Kato2}. They are stable under base change in the category of fs log schemes \cite[Corollary 4.8]{Niziol} and under composition \cite[Corollary 4.11]{Niziol}.


\begin{lemma} \label{one}
Let $F$ be a smooth quasi-compact Kato fan. Let $m$ be a positive integer. Then there exists
a smooth, proper subdivision $\varphi: F'\to F$ such that for every point
$f\in F(\NN)_{\leq m}$ the induced point $\varphi_*^{-1} f \in F'(\NN)$ lies in $F'(\NN)_{\leq 1}$. 
\end{lemma}
\begin{proof}
It is enough to iterate the barycentric subdivision ${m-1}$ times.
\end{proof}

\subsubsection{Weak toroidalisation} We now briefly recall the \emph{weak toroidalisation theorem} of Abramovich--Karu in the language of log geometry; this will be a major tool for us in \S \ref{sec:surjectivity}, and figured already prominently in \cite{Denefbis} and \cite{Denef}.

\begin{theorem}[Abramovich--Karu, Denef] \label{ADK} 
Let $f: X \to Y$ be a dominant morphism of integral varieties
over a field $k$ of characteristic zero.

Then there exist a dominant morphism $f': X' \to Y'$ of smooth integral $k$-varieties,
proper birational morphisms $\alpha_X: X' \to X$ and $\alpha_Y: Y' \to Y$, 
and strict normal crossings divisors $D'\subset X'$ and $E'\subset Y'$ such that
 \begin{itemize}
\item[(1)] the diagram 
$$\xymatrix{
X'\ar[r]^{\alpha_X}\ar[d]_{f'}& X\ar[d]^f\\
Y'\ar[r]^{\alpha_Y}& Y}$$
commutes;
\item[(2)] $f'$ induces
a log smooth morphism of log regular schemes $(X',D') \to (Y',E')$;
\item[(3)] $(f')^{-1}(Y' \setminus E') = X' \setminus D'$.
\end{itemize}
\end{theorem}



This result was first proved over $\mathbf{C}$ in \cite{AbramovichKaru}  and subsequently over arbitrary fields of characteristic zero in \cite{AbramovichDenefKaru}. For an alternative treatment by Gabber and Illusie--Temkin in a more general setting, see \cite[\S 3.8]{illusie}. 

We work with Zariski log schemes instead of \'etale log schemes; Theorem \ref{ADK} is true in both settings, but the statement for Zariski log schemes is slightly stronger. This corresponds to the fact that in \cite[Theorem 1.1]{AbramovichDenefKaru}, the toroidal embeddings can be taken to be \emph{strict}. 
The following example illustrates the difference between both settings:

\begin{example} Let $k$ be a field of characteristic zero such that $-1$ is not a square in $k$. Consider the conic bundle $X\subset \PP^2_k\times{\bf A}^1_k$ given by
$x^2 + y^2 = tz^2$.
Let $\pi:X\to{\bf A}^1_k$ be the natural projection
to the coordinate $t$. Here the total space $X$ is smooth, but the fibre $\pi^{-1}(0)$ is irreducible and singular. Hence, when equipped with the log structure induced by $\pi^{-1}(0)$, the log scheme $(X,\pi^{-1}(0))$ is \emph{not} log regular in our sense, i.e.~as a Zariski log scheme (it is however log regular as an \emph{\'etale} log scheme). 

However, the normalisation of $\pi^{-1}(0)$ is smooth.
If $\psi: X' \to X$ is the blow-up of $X$ at the point $x=y=t=0$, $z=1$, then $(X',(\pi\circ\psi)^{-1}(0))$ is log smooth, and $(X',(\pi\circ\psi)^{-1}(0)) \to ({\bf A}^1_k,0)$ satisfies the requirements of Theorem \ref{ADK}.
\end{example}

\subsection{Logarithmic Hensel's lemma}  In \cite[\S 3.2]{Denefbis}, Denef proved a logarithmic version of Hensel's lemma. We will present a reformulation of this result, with a different proof, written down by Cao in his unpublished MSc thesis \cite{Cao}. 
 
\begin{proposition} \label{logHensel}Let $f: (X,\mathcal{M}_X) \to (Y,\mathcal{M}_Y)$ be a log smooth morphism of fs log schemes. Let $R$ be a complete discrete valuation ring. Let $S = \mathrm{Spec}\,R$ and let $j: s \hookrightarrow S$ be the inclusion of the closed point. Given a commutative diagram 
$$\xymatrix{s^\dagger \ar[r]^{u\ \ \ } \ar[d]_{j} & (X,\mathcal{M}_X) \ar[d]^{f} \\ 
S^\dagger \ar[r]^{t \ \ \ } \ar@{.>}[ur]^{g}& (Y,\mathcal{M}_Y)}$$
of fs log schemes, there is a morphism $g: S^\dagger \to (X,\mathcal{M}_X)$ of fs log schemes such that $gj = u$ and $fg = t$.


\end{proposition}  
 
Here $S^\dagger$ denotes the scheme $S$ equipped with the divisorial log structure induced by the closed point $s$. If $\pi$ is a uniformiser of $R$, a chart for the log structure is given by the map $\mathbf{N} \to R$ which sends $1$ to $\pi$. Similarly, $s^\dagger$ denotes the log point, i.e.~$\mathrm{Spec}\,k$ equipped with the pullback of the log structure on $S^\dagger$ to $s$ (with a chart $\mathbf{N} \to k$ sending 1 to $0$).

\begin{proof} Let $\mathfrak{m}=(\pi)$ be the maximal ideal of $R$.
Write $R_n = R/\mathfrak{m}^n$ and $S_n = \mathrm{Spec}\,R_n$. 
In particular, $s=S_1$.
We have strict closed immersions 
$$i_n: S_n^\dagger \to S^\dagger,\quad i_{n - 1,n}: S_{n - 1}^\dagger \to S_{n}^\dagger$$ 
for all $n \geq 1$. The definition of a log smooth morphism in terms of infinitesimal liftings \cite[\S 3.2, \S 3.3]{Kato1} implies that for any $n \geq 1$ one can find a morphism $$g_n: S_n^\dagger \to (X,\mathcal{M}_X)$$ such that $g_1 = u$ and such that for any $n \geq 1$ both triangles in the diagram
$$\xymatrix{S_{n - 1}^\dagger \ar[r]^{g_{n - 1}} \ar[d]_{i_{n - 1,n}} & (X,\mathcal{M}_X) \ar[d]^{f} \\ S_{n}^\dagger \ar[r]^{t \circ i_n} \ar@{.>}[ur]^{g_n} & (Y,\mathcal{M}_Y)}$$
commute. Since the ring $R = \varinjlim R_n$ is complete, the morphisms $g_n$ induce a well-defined morphism of schemes $g: S \to X$. It is then easy to see that $g$ actually defines a morphism of \emph{log} schemes $g: S^\dag \to (X,\mathcal{M}_X)$ satisfying all requirements. 
\end{proof}

\subsection{Log smooth morphisms and irreducible components} The goal of this section is to prove a basic result (probably well-known to experts) concerning the variation of the (geometric) irreducible components of the fibres of a proper, log smooth morphism of log regular schemes.

If $f: (X,\mathcal{M}_X) \to (Y,\mathcal{M}_Y)$ is a log smooth morphism of log regular schemes, then $f$ induces a morphism of the associated Kato fans $F(f): F(X) \to F(Y)$. The morphisms $f$ and $F(f)$ are compatible with the characteristic morphisms $\pi_X: (X,\mathcal{M}_X^\sharp) \to F(X)$ and $\pi_Y: (X,\mathcal{M}_Y^\sharp) \to F(Y)$. Hence, given $x \in F(X)$ and $y = F(f)(x) \in F(Y)$, we get induced morphisms $U(x) \to U(y)$ and $\overline{U}(x) \to \overline{U}(y)$.

We recall the localisation procedure in \cite[\S 7]{Kato2}. If $(X,\mathcal{M}_X)$ is a log regular scheme, then the ``boundary'' $\partial U(x) = \overline{U}(x) \setminus U(x)$ of the closed subscheme $\overline{U}(x)$ is a divisor, inducing a log structure (cf.~Example \ref{example:divisorial}). Kato proves that the resulting log scheme $(\overline{U}(x),\partial U(x))$ is again log regular; 
the locus of triviality of this log structure is $U(x)$.

The following lemma establishes a ``relative'' version of this localisation procedure; we are not aware of a published account of this basic statement.

\begin{lemma} \label{localisation} Let $f: (X,\mathcal{M}_X) \to (Y,\mathcal{M}_Y)$ be a log smooth morphism of log regular schemes. Let $x \in F(X)$ and let $y = F(f)(x) \in F(Y)$. Then \begin{equation} \label{localized} (\overline{U}(x),\partial U(x)) \to (\overline{U}(y), \partial U(y)) \end{equation} is again a log smooth morphism of log regular schemes. \end{lemma}

\begin{proof} The log regularity is taken care of by \cite[Proposition 7.2]{Kato2}; what we really need to prove is log smoothness. To do so, we will use Kato's criterion (Theorem \ref{theorem:Kato}).

Choose an affine open $W = \mathrm{Spec}\,B$ of $Y$ containing $y$ and a compatible affine scheme $V = \mathrm{Spec}\,A$, equipped with an \'etale map $g: V \to X$ such that $g(V)$ contains $x$, on which there exists a chart for $(V, g^* \mathcal{M}_X) \to (W,\mathcal{M}_W)$ given by homomorphisms of fs monoids $$P \to A,\ Q \to B\ \text{ and } \varphi: Q \to P$$ compatible with the homomorphism of rings $B \to A$. We can assume that the monoid $Q$ is toric, i.e.~that $Q^{\mathrm{gp}}$ is torsion free \cite[Lemma 1.6]{Kato2}. Kato's criterion says that we can then choose $P$ such that the following conditions are satisfied: \begin{itemize}
\item[(1)] the induced morphism $V \to W \times_{\mathrm{Spec}\ \mathbf{Z}[Q]} \mathrm{Spec}\ \mathbf{Z}[P]$ is classically smooth, and
\item[(2)] the induced homomorphism of abelian groups $\varphi^\mathrm{gp}: Q^\mathrm{gp} \to P^\mathrm{gp}$ is injective, and the torsion part of the cokernel has order invertible on $V$.
\end{itemize}

Localising further if needed, we may assume that $x$ has a unique preimage $v$ under $g$. Now $v$ corresponds to an ideal $\mathfrak{p} \in \mathrm{Spec}\, P$, and $y$ corresponds to $\mathfrak{q} = \varphi^{-1}(\mathfrak{p}) \in \mathrm{Spec}\, Q$. Moreover, $g^{-1}(\overline{U}(x)) = \mathrm{Spec}\ A/(\mathfrak{p})$, with log structure given by $P \setminus \mathfrak{p} \to A/(\mathfrak{p})$ induced by the map $P \to A$. Here $(\mathfrak{p})$ is the ideal of $A$ generated by the elements of $\mathfrak{p}$ via $P \to A$. Similarly, $\overline{U}(y) \cap W = \mathrm{Spec}\ B/(\mathfrak{q})$, with log structure given by $Q \setminus \mathfrak{q} \to B/(\mathfrak{q})$. We have an induced morphism $\widetilde{\varphi}: Q \setminus \mathfrak{q} \to P \setminus \mathfrak{p}$ giving a chart for the morphism $$\left(g^{-1}(\overline{U}(x)),g^{-1}(\partial \overline{U}(x))\right) \to (\overline{U}(y) \cap W, \partial \overline{U}(y) \cap W).$$ It suffices to check that this chart satisfies Kato's criterion, i.e.~that

\begin{itemize}
\item[(1$'$)] the induced morphism $$\mathrm{Spec}\,A/(\mathfrak{p}) \to \mathrm{Spec}\,B/(\mathfrak{q}) \times_{\mathrm{Spec}\ \mathbf{Z}[Q \setminus \mathfrak{q}]} \mathrm{Spec}\ \mathbf{Z}[P \setminus \mathfrak{p}]$$ is classically smooth;
\item[(2$'$)] the induced homomorphism of abelian groups $\widetilde{\varphi}^\mathrm{gp}:(Q \setminus \mathfrak{q})^\mathrm{gp} \to (P \setminus \mathfrak{p})^\mathrm{gp}$ is injective, and the torsion part of the cokernel has order invertible on $g^{-1}(\overline{U}(x))$.
\end{itemize}

Now (1$'$) follows immediately from the fact that smooth morphisms are stable under base change, since $V \to W \times_{\mathrm{Spec}\ \mathbf{Z}[Q]} \mathrm{Spec}\ \mathbf{Z}[P]$ is already smooth. 

Concerning (2$'$), the injectivity of $\widetilde{\varphi}^\mathrm{gp}$ is trivial.  Consider the commutative diagram of short exact sequences

$$
\xymatrix{
0\ar[r]& (Q \setminus \mathfrak{q})^\mathrm{gp}  \ar[r] \ar@{^{(}->}[d]&
(P \setminus \mathfrak{p})^\mathrm{gp} \ar[r]\ar@{^{(}->}[d]&
\mathrm{coker}\ \widetilde{\varphi}^\mathrm{gp} \ar[r]\ar[d]&0\\
0\ar[r]& Q^\mathrm{gp}  \ar[r] & P^\mathrm{gp} \ar[r]&
\mathrm{coker}\ \varphi^\mathrm{gp} \ar[r]&0}$$
The fact that the order of the torsion part of the cokernel of $\widetilde{\varphi}^\mathrm{gp}$ is invertible on $g^{-1}(\overline{U}(x))$ follows from the same statement for the cokernel of $\varphi^\mathrm{gp}$, together with the snake lemma and the observation that the cokernel of $(Q \setminus \mathfrak{q})^\mathrm{gp} \to Q^\mathrm{gp}$ is torsion free -- indeed, the fs monoid $Q$ is toric, and $Q \setminus \mathfrak{q}$ is one of its faces. This finishes the proof. \end{proof}

The following lemma concerns the fibres of log smooth families over a base with trivial log structure. In this setting, the functor $\Irr_{X/Y}$  from \S\ref{sec:Irr} becomes rather simple.

\begin{lemma} \label{stein} Let $f: (X,\mathcal{M}_X) \to (Y,\OO_Y^*)$ be a proper, log smooth morphism of log regular schemes, where $Y$ is given the trivial log structure.
Then $f:X\to Y$ is flat, and $\Irr_{X/Y}$ is represented by a finite \'etale scheme over $Y$.  \end{lemma}
\begin{proof} 
Since the log structure on $Y$ is trivial, $f$ is integral \cite[Corollary 4.4(ii)]{Kato1}. Log smooth, integral morphisms are flat \cite[Corollary 4.5]{Kato1}, hence $f$ is flat.

The fibres of $f$ are log regular, being log smooth over a field with trivial log structure \cite[Theorem~8.2]{Kato2}, and hence normal. Moreover, as log smooth morphisms are preserved by base change, the fibres are geometrically normal, hence geometrically reduced.

Let $X \to \mathrm{St}_{X/Y} \to Y$ be the Stein factorisation of $X \to Y$. As $f$ is flat and the fibres are geometrically reduced, it follows from \cite[Tag 0BUN]{Stacksproject} that $\mathrm{St}_{X/Y}$ is finite \'{e}tale over $Y$. The result now follows from the fact that $\Irr_{X/Y}$ and $\mathrm{St}_{X/Y}$ are isomorphic under the assumptions of the lemma. Indeed, the fibres of $f$ are normal, so the connected components of the geometric fibres are precisely the irreducible components.
\end{proof}

Another important property of proper, log smooth morphisms is that the multiplicities of the fibres are ``constant along strata''. Such a result is well-known (and easy to check) for toric morphisms, and readily extends to the log smooth setting.

\begin{proposition} \label{proposition:constancy} Let $f: (X,\mathcal{M}_X) \to (Y,\mathcal{M}_Y)$ be a proper, log smooth morphism of log regular schemes. Let $y' \in Y$, and let $y \in F(Y)$ be the unique point with the property that $y' \in U(y)$. Let $x'$ be one of the generic points of the fibre $f^{-1}(y')$, and let $x \in F(X)$ be the unique point such that $x' \in U(x)$.

Then $f(x) = y$ and the multiplicity of $x$ in $f^{-1}(y)$ is equal to the multiplicity of $x'$ in $f^{-1}(y')$, that is, the lengths of the Artinian local rings $\mathcal{O}_{f^{-1}(y),x}$ and $\mathcal{O}_{f^{-1}(y'),x'}$ are equal. \end{proposition} 

\begin{proof} Choose a chart for $f$ around $x'$. This involves choosing

\begin{itemize}
\item a Zariski open $j:V \hookrightarrow Y$ such that $y'\in V$ together with a chart 
$c_V: V \to \mathrm{Spec}\,\mathbf{Z}[Q]$ for $j^* \mathcal{M}_Y$,
\item an \'etale morphism $h: U \to X$ such that $x'\in h(U)$ and $(f\circ h)(U)\subset V$
together with a chart $c_U: U \to \mathrm{Spec}\,\mathbf{Z}[P]$ for $h^* \mathcal{M}_X$,
\item a  homomorphism $\varphi: Q \to P$ compatible with $f \circ h: U \to V$, $c_U$ and $c_V$,
\end{itemize}
such that the induced map $$\psi: U \to V \times_{\mathrm{Spec}\,\mathbf{Z}[Q]} \mathrm{Spec}\,\mathbf{Z}[P]$$ is \'etale. Let \begin{eqnarray*} \pi_1: V \times_{\mathrm{Spec}\,\mathbf{Z}[Q]} \mathrm{Spec}\,\mathbf{Z}[P] & \to & V \\ \pi_2: V \times_{\mathrm{Spec}\,\mathbf{Z}[Q]} \mathrm{Spec}\,\mathbf{Z}[P] & \to & \mathrm{Spec}\,\mathbf{Z}[P]\end{eqnarray*} be the natural projections.
Choose a point $z'\in h^{-1}(x')$. 

Since $h$ and $\psi$ are \'etale, the lengths of the Artinian local rings $$\mathcal{O}_{f^{-1}(y'),x'},\ \mathcal{O}_{(f \circ h)^{-1}(y'),z'}\ \textrm{ and }\ \mathcal{O}_{\pi_1^{-1}(y'),\psi(z')}$$ are equal. Since the projection $\pi_1$ is obtained by base change from the toric morphism $\mathrm{Spec}\,\mathbf{Z}[\varphi]: \mathrm{Spec}\,\mathbf{Z}[P] \to \mathrm{Spec}\,\mathbf{Z}[Q]$, the length of the local ring $\mathcal{O}_{\pi_1^{-1}(y'),\psi(z')}$ 
only depends on the toric strata of $\mathrm{Spec}\,\mathbf{Z}[P]$ and $\mathrm{Spec}\,\mathbf{Z}[Q]$ which contain $\pi_2(\psi(z')) = c_U(z')$ and $c_V(y') = (\mathrm{Spec}\,\mathbf{Z}[\varphi])(c_U(z'))$, respectively, i.e. on the prime ideals of $P$ and $Q$ which correspond to $x$ and to $y$, respectively. This gives the desired result. \end{proof}

We will need the following elementary lemma.

\begin{lemma} \label{lem:remove_Z}
	Let $f: X \to Y$ be a morphism of schemes of finite presentation. 
	Let $Z \subset X$ be a closed
	subset such that $Z|_{f^{-1}(y)}$ contains no irreducible component of $f^{-1}(y)$,
	for all points $y \in Y$. Then the natural morphism of functors
	$$\Irr_{(X \setminus Z)/Y} \to \Irr_{X/Y}$$
	is an isomorphism.
\end{lemma}
\begin{proof}
	Our assumptions imply that removing $Z$ does not change the open irreducible
	components of the fibres; the result follows. (See also \cite[Corollary 2.6.2]{Rom11}.)
\end{proof}

From the previous results we will now deduce the final statement of this section.

\begin{proposition} \label{5.16} Let $f: (X,\mathcal{M}_X) \to (Y,\mathcal{M}_Y)$ be a proper, log smooth morphism of log regular schemes. For each point $y \in F(Y)$,   the functors $$\Irr_{f^{-1}(U(y))/U(y)}\ \textrm{ and }\ \Irr^1_{f^{-1}(U(y))/U(y)}$$ are representable by finite \'etale schemes over $U(y)$. \end{proposition}

\begin{proof} The set $f^{-1}(U(y))$ is the disjoint union of locally closed subsets $U(x)$,
where $x\in F(X)$ is such that $f(x)=y$. For each such $x$
consider the induced morphism $$(\overline{U}(x),\partial U(x)) \to (\overline{U}(y),\partial U(y)).$$ 
By Lemma \ref{localisation}, this is again a proper and log smooth morphism of log regular schemes.
Let $f_x$ be the restriction of this morphism to $U(y)$ equipped with the trivial log structure. 
The scheme $\overline{U}(x) \times_{\overline{U}(y)} U(y)$ has a natural log structure, obtained from the divisorial log structure on $(\overline{U}(x),\partial U(x))$ by restriction, making
$$f_x:\overline{U}(x) \times_{\overline{U}(y)} U(y) \to U(y)$$ 
a proper and log smooth morphism of log regular schemes. The log structure on $U(y)$ is trivial, so
by Lemma \ref{stein} the morphism $f_x$ is flat.
The fibres of a proper flat morphism of irreducible varieties have the same pure dimension; therefore all fibres of $f_x$ have pure dimension $\dim(U(x))-\dim(U(y))$. 

The stratum $U(x)$ is a dense open subset of $\overline{U}(x)\times_{\overline{U}(y)} U(y)$,
whose complement is the union of closed subsets $\overline U(x')\subset\overline{U}(x)$, where $x'\in F(X)$ is a specialisation of $x$, $f(x')=y$ and $x' \neq x$. In particular, $\dim(U(x'))<\dim(U(x))$.
Thus for any $t\in U(y)$ the intersection $f_x^{-1}(t)\cap U(x)\subset f_x^{-1}(t)$ is the complement
of the union of the closed subsets $f_x^{-1}(t)\cap\overline U(x')=f_{x'}^{-1}(t)$, with $x'$ as above.
We have seen that $f_{x'}^{-1}(t)$ has pure dimension $$\dim(U(x'))-\dim(U(y))<\dim(U(x))-\dim(U(y)),$$
hence $f_x^{-1}(t)\cap U(x)$ is a dense open subset of $f_x^{-1}(t)$. 
It now follows from Lemma \ref{lem:remove_Z} there is natural isomorphism of functors
\begin{equation}\Irr_{U(x)/U(y)}=\Irr_{\overline U(x)\times_{\overline{U}(y)} U(y)/U(y)}.
\label{13july}
\end{equation}

Now let $x_1,\ldots,x_n \in F(X)$ be the minimal elements of $f^{-1}(y)$ with respect to the partial ordering given by the topology on $F(X)$. Then $f^{-1}(U(y))$ is the union of the closed subsets $\overline{U}(x_i)\times_{\overline{U}(y)} U(y)$,
where $i=1,\ldots,n$. Hence the fibre $f^{-1}(t)$
is the union of the closed subsets $f_{x_i}^{-1}(t)$. For each $i=1,\ldots,n$, the intersection
$f_{x_i}^{-1}(t)\cap U(x_i)$ is open and dense in $f_{x_i}^{-1}(t)$, hence
$\bigcup_{i=1}^n (f_{x_i}^{-1}(t)\cap U(x_i))$ is open and dense in $f^{-1}(t)$. 
Therefore Lemma \ref{lem:remove_Z} yields a natural isomorphism of functors
$$\Irr_{f^{-1}(U(y))/U(y)}=\Irr_{\big(\bigcup_{i=1}^n U(x_i)\big)/U(y)}=
\coprod_{i=1}^n \Irr_{U(x_i)/U(y)}$$
as the $U(x_i)$ are pairwise disjoint.

By (\ref{13july})
the result for the functor $\Irr_{f^{-1}(U(y))/U(y)}$ now follows from Lemma \ref{stein}. Next, $\Irr^1_{f^{-1}(y)/y}=\Irr_{f^{-1}(y)_{\rm sm}/y}$ represents 
the irreducible components of $f^{-1}(y)$ of geometric multiplicity 1.
By Proposition \ref{proposition:constancy}, the functor $\Irr^1_{f^{-1}(U(y))/U(y)}
=\Irr_{f^{-1}(U(y))_{\rm sm}/U(y)}$ represents
the Zariski closure of its generic fibre, which is represented by $\Irr^1_{f^{-1}(y)/y}$. 
This is the union of some
of the irreducible components of the finite \'etale $U(y)$-scheme represented by
$\Irr_{f^{-1}(U(y))/U(y)}$, and hence
is also finite and \'etale over $U(y)$.
\end{proof}

\section{Surjectivity} \label{sec:surjectivity}

In this section we prove Theorem \ref{thm:CT}, and 
explain how to use this and other results from the paper to
prove Theorems \ref{maintheorem} and \ref{cor:frobenian}.

\subsection{Notation and hypotheses} Let $f: X \to Y$ be a dominant morphism of 
smooth, proper, geometrically integral varieties over 
a number field $k$ with geometrically integral generic fibre.
By Lemma~\ref{lem:Denef_obs}  and Theorem \ref{ADK}
we can assume that $f$
gives rise to a log smooth morphism $f:(X,D)\to (Y,E)$ of log regular schemes. 

Let $U = X \setminus D$ and $V = Y \setminus E$. 
By condition (3) of Theorem \ref{ADK}
we have $f^{-1}(V) = U$. Thus 
$f:U\to V$ is a smooth proper morphism with geometrically integral generic fibre.
By \cite[Corollaire 15.5.4]{EGAIV} all fibres of $f:U\to V$ are geometrically connected.
Since connected noetherian normal schemes are integral
\cite[Lemma 27.7.6, Tag 033H]{Stacksproject}, all fibres of $f:U\to V$ are
geometrically integral.

Recall that we are working with Zariski log structures; in particular,
$D$ and $E$ are strict normal crossing divisors, so their irreducible components are smooth. Let $F(f):F_X\to F_Y$ be the attached morphism of smooth Kato fans. We have the height $h_Y:F_Y(\NN)\to\NN$ introduced at the end of \S \ref{subsubsection:fans}.

Let $E_i$ be an irreducible component of $E$
and let $D_j$ be an irreducible component of $D$. Since $X$ and $Y$ are smooth,
each of the local rings $\O_{E_i,Y}$ and $\O_{D_j,X}$ is a discrete valuation ring. 
Let $\val_{E_i}:\O_{E_i,Y}\to\NN$ and $\val_{D_j}:\O_{D_j,X}\to\NN$
be the respective discrete valuations.
There is an $m_{ij}\in\NN$
such that the restriction of $\val_{D_j}$ to $\O_{E_i,Y}$ is $m_{ij}\val_{E_i}$.

Let $S$ be a finite set of primes of $k$
such that $f:X\to Y$ extends to a morphism of smooth and proper $\O_S$-schemes
$f:\sX\to\sY$. For each irreducible component $D'\subset D$ let
$k'$ be the algebraic closure of $k$ in the field of functions 
$\kappa(D')$. We assume that $S$ contains all the primes of $k$ ramified in 
any of the fields $k'$. 

Let $v$ be a prime of $k$ which is not in $S$. We 
denote by $\val_v:\O_v\setminus\{0\}\to\ZZ$ the valuation at $v$ and
by $\pi_v$ a generator of the maximal ideal of $\O_v$, so that $\F_v=\O_v/(\pi_v)$. 
We write $\sY_v=\sY\times_{\Spec\,O_S}\Spec\,O_v$
and $\sE_v=\sE\times_{\Spec\,O_S}\Spec\,O_v$.

Let $\sD\subset\sX$ and $\sE\subset\sY$ be 
the Zariski closures of $D$ and $E$ in $\sX$ and $\sY$, respectively.
By adding finitely many primes to $S$ we can assume that 
$\sE$ and $\sD$ are strict normal crossing divisors, so each of their
irreducible components is smooth over $\Spec\,\O_S$. 
We can also assume that
$(\sX,\sD)$ and $(\sY,\sE)$ are log regular schemes with
Kato fans $F_X$ and $F_Y$, respectively, and that there is a log smooth morphism
$f:(\sX,\sD)\to(\sY,\sE)$ inducing the given morphism $(X,D)\to(Y,E)$ on the generic fibres.

By the valuative criterion of properness any point 
$P\in Y(k_v)$ extends to a local section $\sP:\Spec\,\O_v\to \sY$
of the structure morphism $\sY\to \Spec\,\O_S$.
We write $P\bmod\pi_v$ for $\sP(\Spec\,\F_v)$, which is an $\F_v$-point of
the closed fibre $\sY\times_{\Spec\,O_S}\Spec\,\F_v$.

For any point $P\in V(k_v)$ the local section $\sP:\Spec\,\O_v\to \sY$ gives rise to a morphism of log schemes
$\sP:(\Spec\,\O_v)^\dagger\to (\sY,\sE),$
where $(\Spec\,\O_v)^\dagger$ is equipped with the natural (divisorial) log structure. The associated morphism of Kato fans
is $F(\sP):\Spec\,\NN\to F_Y$. It sends the closed point $\NN_{>0}$
of $\Spec\,\NN$ to $F(P \bmod\pi_v)$.

\subsection{A surjectivity criterion} The following intermediate result is an adaptation of Denef's ``surjectivity criterion"
\cite[4.2]{Denef} in the language of log geometry.

\begin{proposition} \label{p0} \label{proposition:surjectivity}
There exists an integer $m \geq 1$ such that whenever  $v\notin S$,
we have $f(X(k_v))=Y(k_v)$ if and only if $f(X(k_v))$ contains all 
$P\in V(k_v)$ with $h_Y(F(\sP))\leq m$.
\end{proposition}
\begin{proof} 
For each $s\in F_X$ let $F_X^s$ be the open Kato subcone of $F_X$ defined by $s$.
Similarly, for each $t\in F_Y$ let $F_Y^t $ be the open Kato subcone of $F_Y$ defined by $t$.
The morphism $F(f):F_X\to F_Y$ induces a map
$F(f)_*:F_X(\NN)\to F_Y(\NN)$. For each $t\in F_Y$ define
$$m_t={\rm min}\{h_Y(r)\ |\ r\in F_Y^t(\NN),\ r\notin F(f)_*(F_X(\NN))\}.$$
(If the set in the right hand side is empty, we take $m_t=0$.)

For each $s\in F_X$ mapping to $t\in F_Y$ define
$$m_{s,t}={\rm min}\{h_Y(r)\ |\ r\in F(f)_*(F_X^s(\NN))\subset F_Y^t(\NN)\}.$$
Finally, let
$$m={\rm max}\{m_t, m_{s,t}\},$$
where $t\in F_Y$ and where $s \in F_X$ maps to $t \in F_Y$.

Let us prove that $m$ satisfies the conclusion of the proposition.
One implication is obvious, so we prove the other one: if
every  $P\in V(k_v)$ with $h_Y(F(\sP))\leq m$ is contained in $f(X(k_v))$,
we need to show that $Y(k_v)=f(X(k_v))$.  

It suffices to show that $V(k_v) \subseteq f(U(k_v))$. Indeed, for the topology of
$k_v$ the set $V(k_v)$
is dense $Y(k_v)$, whereas $f(X(k_v))$ is closed in $Y(k_v)$ because $f$ is proper.

Take any $P\in V(k_v)$. There exists an $r\in F_Y(\NN)$, $h_Y(r)\leq m$,
such that the image of $F(\sP)$ and $r$ are both contained in some $F_Y^t$ and, furthermore,
either both are in the image of some $F_X^s(\NN)$ or neither is contained in
$F(f)_*(F_X(\NN))$. 

Let $(E_i)_{i \in I}$ be the irreducible components of $E$
such that the stratum $U(t)$ defined by $t\in F_Y$ is a connected component of the locally closed subscheme (with reduced structure)
$$\left(\bigcap_{i\in I}E_i\right)\setminus \left(\bigcup_{i\notin I}E_i\right)\subset Y.$$
Write $r=\sum_{i\in I}r_i u_i$, where $r_i$ is a non-negative integer and $u_i \in F_Y(\mathbf{N})$ is the primitive generator of the one-dimensional cone corresponding to $E_i$, so that $r_i\not=0$ if and only if $i\in I$. The height of $r$ is then given by the formula
$h_Y(r)=\sum_{i\in I} r_i$.

We now construct a commutative diagram of log schemes
\begin{equation}\xymatrix{
(\Spec\,\F_v)^\dagger\ar[r]\ar[d] & (\Y_v,\sE_v) \ar[d]\\
(\Spec\,\O_v)^\dagger\ar[r] & (\Spec\,\O_v)^{\rm tr}}\label{uno}
\end{equation}

Here $(\Spec\,\O_v)^{\rm tr}$ stands for the scheme $\Spec\,\O_v$ equipped with the trivial log structure $\O_v^* \to O_v$.
The morphism $(\Spec\,\O_v)^\dagger\to (\Spec\,\O_v)^{\rm tr}$
is the forgetful morphism defined by the identity morphism of the underlying schemes
and by the natural morphism of monoids $\O_v^*\to \O_v\setminus\{0\}$.
The right hand vertical arrow is induced by the structure morphism $\Y_v\to \Spec\,\O_v$. 
Next, $(\Spec\,\F_v)^\dagger$ is the standard log point defined by the monoid 
$(\O_v\setminus\{0\})/(1+\pi_v\O_v)$. The left hand vertical arrow is the natural
morphism of log schemes. On the underlying schemes
the top horizontal arrow sends $\Spec\,\F_v$ to $P\bmod\pi_v$.
The commutativity of (\ref{uno}) as a diagram of schemes is clear.

We need to define the top horizontal arrow as a morphism of log schemes.
Let $A$ be an $\O_v$-algebra such that
$\Spec\,A$ is an affine neighbourhood of $P\bmod\pi_v$ in $\sY_v$.
We can assume that $A$ contains a local equation
$\pi_i$ for the Zariski closure of $E_i$ in $\sY_v$, where $i\in I$.
Let $S\subset A$ be the multiplicative system generated by the $\pi_i$
and let $S^{-1}A$ be the localisation of $A$ with respect to $S$.
The log scheme $(\sY,\sE)$ is defined by the subsheaf of $\O_\sY$
consisting of functions invertible outside $\sE$. Hence 
the log scheme $(\Spec A,A\cap (S^{-1}A)^*)$ is an open subscheme of $(\sY,\sE)$.
Locally the diagram of monoids attached to (\ref{uno}) is
\begin{equation}\xymatrix{
(\O_v\setminus\{0\})/(1+\pi_v\O_v)& A\cap (S^{-1}A)^*\ar[l]\\
\O_v\setminus\{0\}\ar[u]^{\varphi}& \O_v^*\ar[l]\ar[u]}\label{uno_monoid}
\end{equation}

Here $\varphi$ is the canonical surjective morphism of monoids\footnote{This is related to 
Denef's notion of {\em multiplicative residue} \cite[Definition 3.2]{Denef}.
The projection to $\F_v^*$ is called the {\em angular component} in
\cite[Definition 3.5]{Denef}.}
$$\varphi:\O_v\setminus\{0\}\lra (\O_v\setminus\{0\})/(1+\pi_v\O_v)\cong\F_v^*\oplus\NN,$$
where the isomorphism depends on the choice of $\pi_v$. Let us denote by $\ov\varphi$
the composition of $\varphi$ with the projection to $\F_v^*$.
The choice of $\pi_i$, $i\in I$, gives an isomorphism 
$$A\cap (S^{-1}A)^*\cong A^*\times\NN^I.$$
To complete the definition of $(\Spec\,\F_v)^\dagger\to (\Y_v,\sE_v)$ we define
$$A^*\times\NN^I \lra \F_v^*\oplus\NN$$
as the morphism of monoids that sends $\alpha\in A^*$ to $(\alpha(P \bmod\pi_v), 0)$
and sends the canonical generator $1_i\in\NN^I$ to 
$(\varphi(\pi_i(\sP)), r_i).$
Checking the commutativity of the diagram (\ref{uno_monoid}) is straightforward.
The factor $\ov\varphi(\pi_i(\sP))$ is irrelevant for the commutativity of this diagram,
but will play a role later in the proof.

Since $\Y\to\Spec\,\O_S$ is smooth,
the morphism of log schemes $(\Y,\sE)\to (\Spec\,\O_v)^{\rm tr}$ is log smooth.
Applying Proposition \ref{logHensel}
to (\ref{uno}) produces a 
morphism $(\Spec\,\O_v)^\dagger\to (\Y,\sE)$ such that the two
resulting triangles commute. This gives a point $Q\in V(k_v)$ such that
$$Q\bmod\pi_v=P\bmod\pi_v,\quad  F(\sQ)=r, \quad
\ov\varphi(\pi_i(\sP))=\ov\varphi(\pi_i(\sQ)), \quad i\in I.$$
Since $h_Y(F(\sQ))\leq m$, our assumptions imply that $Q=f(R)$
for some $R\in U(k_v)$. In particular, $P\bmod\pi_v=f(R\bmod\pi_v)$
and $F(\sP)=F(f)_*(a)$ for some $a\in F_X^s(\NN)$, where $s\in F_X$ maps to $t\in F_Y$
and $F(\sR)\in F_X^s(\NN)$. Let $(D_j)_{j \in J}$ be the irreducible components of $D$
such that the stratum $U(s)$ defined by $s\in F_X$ is a connected component of
$$\left(\bigcap_{j\in J}D_j\right)\setminus \left(\bigcup_{j\notin J}D_j\right)\subset X.$$
Write $a=\sum_{j\in J}a_j u'_j$, where $a_j$ is an integer and $u'_j \in F_X(\mathbf{N})$ is the primitive generator of the one-dimensional cone corresponding to $D_j$, so that $a_j\not=0$ if and only if $j\in J$. 

Let us now construct a commutative diagram of log schemes
\begin{equation}\xymatrix{
(\Spec\,\F_v)^\dagger\ar[r]\ar[d] & (\X,\sD) \ar[d]\\
(\Spec\,\O_v)^\dagger\ar[r] & (\sY,\sE).}\label{duo}
\end{equation}
Here the vertical arrows are the obvious morphisms,
and the lower horizontal arrow is given by $\sP
\in \sY_v(\O_v)$. The upper horizontal arrow
sends $\Spec\,\F_v$ to $R\bmod\pi_v$. 
The commutativity of (\ref{duo}) as a diagram of schemes is clear
since $P\bmod\pi_v=f(R\bmod\pi_v)$.

Choose an $A$-algebra $B$
such that $\Spec\,B$ is an affine neighbourhood of $R\bmod\pi_v$ in $\sX_v$.
We can assume that $B$ contains a local equation
$\varpi_j$ for the closure of $D_j$ in $\sX_v$, where $j\in J$. Define
$(\Spec\,\F_v)^\dagger\to(\X,\sD)$
via the morphism of monoids 
$$B\cap(S^{-1}B)^*\cong B^*\oplus\NN^J\lra \F_v^*\oplus\NN$$
such that $\beta\in B^*$ goes to $(\beta(R \bmod\pi_v),0)$ and 
$$1_j\mapsto(\ov\varphi(\varpi_j(\sR)),a_j).$$
To check the commutativity of (\ref{duo}) as a diagram of log schemes we need to check
the commutativity of the diagram of monoids

\begin{equation*}\xymatrix{
\F_v^*\oplus\NN & \ar[l]B^*\oplus \NN^J\\
\O_v\setminus\{0\}\ar[u]^\varphi & \ar[u]\ar[l]A^*\oplus\NN^I.}
\end{equation*}
For $(\alpha,0) \in A^*\oplus\NN^I$ the commutativity follows from $P\bmod\pi_v=f(R\bmod\pi_v)$.
Write $$\pi_i=b_i\prod\varpi_j^{m_{ij}},$$
where $b_i\in B^*$. The right hand vertical map sends the element
$(0,1_i) \in A^* \oplus \NN^I$ to the element $(b_i,\sum_{j\in J} m_{ij}1_j)$ of $B^* \oplus \NN^J$, which then goes to 
\begin{equation}
\left(b_i(R\bmod\pi_v)\prod_{j\in J}\ov\varphi(\varpi_j(\sR))^{m_{ij}}
,\sum_{j\in J} m_{ij}a_j\right).\label{mmm}
\end{equation}
The lower horizontal arrow sends $1_i$ to $\pi_i(\sP)$ 
and then the left hand vertical arrow gives
$(\ov\varphi(\pi_i(\sP)),\val_v(\pi_i(\sP)))$.
This coincides with (\ref{mmm}). Indeed,
on the one hand, $\val_v(\pi_i(\sP))$ is the $i$-th coordinate of $F(\sP)=F(f)_*(a)$,
which equals $\sum_{j\in J} m_{ij}a_j$. On the other hand, $\ov\varphi(\pi_i(\sP))=\ov\varphi(\pi_i(\sQ))$
and $Q=f(R)$. This proves the commutativity of (\ref{duo}).

The morphism $(\sX,\sD)\to (\sY,\sE)$ is log smooth,
hence we can apply Proposition \ref{logHensel} to the diagram (\ref{duo}). We deduce that there is
a morphism $(\Spec\,\O_v)^\dagger\to (\sX,\sD)$ whose composition
with $f:(\sX,\sD)\to(\sY,\sE)$ is the morphism 
$(\Spec\,\O_v)^\dagger\to(\sY,\sE)$ given by $P$.
Therefore $P\in f(X(k_v))$. This proves that $Y(k_v)=f(X(k_v))$. 
\end{proof}

\subsection{Further modifications of the log smooth model} We now use the bound obtained in Proposition \ref{p0} to construct a specific modification of the morphism $f$.

\begin{proposition} \label{p1}
Let $f:X\to Y$ and $m \geq 1$ be as above. 
There exists a commutative diagram of morphisms of log regular schemes
\begin{equation}
\xymatrix{(X',D')\ar[r]^{\sigma_X}\ar[d]_{f'}&(X,D)\ar[d]^{f}\\
(Y',E')\ar[r]^{\sigma_Y}&(Y,E)}\label{p1d}
\end{equation}
where $X'$ and $Y'$ are smooth, proper and geometrically integral, 
$f'$ is dominant and log smooth, 
$\sigma_X$ and $\sigma_Y$ are log blow-ups 
and the following holds: if $P\in V(k_v)$ satisfies $1\leq h_Y(F(\sP))\leq m$,
then $\sigma_Y^{-1}(P)\bmod\pi_v$ is a smooth point of 
the reduction of $E'$.
\end{proposition}

The last statement of the proposition implies that $\sigma_Y^{-1}(P)\bmod\pi_v$ 
belongs to exactly one geometric irreducible component of $E'$.

\begin{proof} We let $\sigma_Y:(Y',E')\to(Y,E)$ be the log blow-up
defined by the subdivision of the Kato fan of $(Y,E)$ as in
Lemma \ref{one}. We see that $Y'$ is smooth and proper and $E'$ is a strict normal crossings divisor (so $(Y',E')$ is log regular).  

Recall that in the category of fs log schemes, log blow-ups are stable under composition \cite[Corollary 4.11]{Niziol} and 
under base change \cite[Proposition 4.5, Corollary 4.8]{Niziol}. 

Indeed, Nizio{\l} shows in
\cite[Corollary 4.8]{Niziol} that the fs fibred product 
of $(X,D)$ and $(Y',E')$ over $(Y,E)$ is a log blow-up $\sigma_X:(X',D')\to (X,D)$ making 
(\ref{p1d}) commute. More precisely, if $(Y',E')$ is the log blow-up of $(Y,E)$ 
in a coherent ideal ${\mathcal I}\subset{\mathcal M}_Y$, then
$(X',D')$ is the log blow-up of $(X,D)$ in the inverse image ideal
${\mathcal I}{\mathcal M}_X$.

We note that $(X',D')$ is a
log regular scheme by \cite[Theorem 8.2]{Kato2}. Since log smooth morphisms
are stable under base change, the morphism $f':(X',D')\to(Y',E')$ is log smooth.
There is a further log blow-up $(X'',D'')\to (X',D')$ 
such that $X''$ is smooth as a scheme (see \cite[(10.4)]{Kato2} or \cite[Theorem 5.8]{Niziol}).
The composition $(X'',D'')\to (Y',E')$ is still log smooth. Therefore
replacing $(X',D')$ by $(X'',D'')$ yields the result.
\end{proof}


Recall that the codimension 1 strata of $(Y',E')$ are precisely the irreducible components
of the smooth locus of $E'$. The morphisms $\sigma_X$ and $\sigma_Y$ in
Proposition \ref{p1} are log blow-ups, so they induce isomorphisms $X'\setminus D'\cong U=X\setminus D$
and $Y'\setminus E'\cong V=Y\setminus E$.

\begin{corollary} \label{cc}
Let $f':(X',D')\to (Y',E')$ be as in Proposition \ref{p1}. Then there exists a finite set $S\subset\Omega_{k,\f}$
such that for $v\notin S$ we have $f(X(k_v))=Y(k_v)$ if and only if  $s_{f',\kappa(Z)}(v)=1$ for each 
codimension $1$ stratum $Z$ of $(Y',E')$.
\end{corollary}
\begin{proof} Let $f':(\sX',\sD')\to(\sY',\sE')$ be a model of $(X',D')\to(Y',E')$ over $\OO_{k,S}$. 
We enlarge the finite set of places $S$ to ensure that
$f':(\sX',\sD')\to(\sY',\sE')$ is a proper, log smooth morphism of log regular schemes over $\OO_{k,S}$
such that the induced morphism of Kato fans is the same as for $(X',D')\to(Y',E')$. 

Let $W=Y'\setminus E'_{\rm sing}$ and let $\W\subset\Y'$ be the complement to the Zariski closure of
$E'_{\rm sing}$ in $\Y'$. For a codimension 1 stratum $Z$ of $(Y',E')$ 
we denote by $\sZ$ the Zariski closure of $Z$ in $\W$. We now further enlarge $S$ to ensure that the following properties hold.

\begin{itemize}
\item[(1)] For any $v\notin S$ and for any $x\in \W(\F_v)$ such that
the fibre $(f')^{-1}(x)$ is split, there is a smooth $\F_v$-point in $(f')^{-1}(x)$. 
This is arranged via the Lang--Weil inequality.
\item[(2)] For each codimension 1 stratum $Z$ of $(Y',E')$, for any $v\notin S$ such that 
$s_{f',\kappa(Z)}(v)=1$ and for any $x\in \sZ(\F_v)$ the fibre $(f')^{-1}(x)$ is split.
This is achieved by applying Corollary \ref{cor:sur} to the morphism $(f')^{-1}(Z)\to Z$.
Indeed, in this case the assumption of this corollary is satisfied by
Proposition \ref{5.16}.
\item[(3)] For any $v\notin S$, if there exists a codimension 1 stratum $Z$ of $(Y',E')$ 
such that $s_{f',\kappa(Z)}(v)<1$,
then $Y(k_v)\not=f(X(k_v))$. This follows from Theorem \ref{thm:not_surjective}.
\end{itemize}

Let us now prove the statement of the corollary. Fix any $v\notin S$. One implication is immediate: if there exists a codimension 1 stratum $Z$ of $(Y',E')$ with
$s_{f',\kappa(Z)}(v)<1$, then (3) implies that $Y(k_v)\not=f(X(k_v))$.

Conversely, assume that $s_{f',Z}(v)=1$ for each codimension 1 stratum $Z$ of $(Y',E')$. 
By Propositions \ref{proposition:surjectivity} and \ref{p1},
it suffices to prove that if $P\in V(k_v)$ is such that $P\bmod\pi_v$
is in $\W\times_{\Spec\,O_S}\Spec\,\F_v$, then $P\in f(X(k_v))$. Write $\U=\X'\setminus\sD'$ and $\V=\Y'\setminus\sE'$.

Let us first consider the case when $P\bmod\pi_v\in\V(\F_v)$.
The morphism $\left.f'\right|_{\mathcal{U}}:\U\to\V$ is smooth and proper, with geometrically integral generic fibre.
This implies that all fibres of $\left.f'\right|_{\mathcal{U}}$ are smooth and geometrically integral; in particular, such is $(f')^{-1}(P\bmod\pi_v)$. By (1) this fibre has a smooth $\F_v$-point, which is clearly also a smooth point of $f'$.

In the case when $P\bmod\pi_v\in\sZ(\F_v)$
for some codimension 1 stratum $Z$ of $(Y',E')$, a combination of (1) and (2) implies
that the fibre $(f')^{-1}(P\bmod\pi_v)$ has a smooth $\F_v$-point. But a smooth point in this fibre is actually a smooth point of the morphism $f'$ as well, since $\left.f'\right|_{(f')^{-1}(\mathcal{W})}$ is a flat morphism, by \cite[Corollary 4.4.(ii), Corollary 4.5]{Kato1}. 

An application of the classical version of Hensel's lemma now allows one to lift such an $\F_v$-point to a $k_v$-point of $X'$ in both of the cases considered above, as required. \end{proof}

\subsection{Proofs of the main theorems} Theorem~\ref{thm:CT} now follows immediately from Corollary \ref{cc}, and 
Theorem \ref{maintheorem} 
follows from Lemma~\ref{lem:is_pseudo-split}, Corollary \ref{non-sur} and Theorem~\ref{thm:CT}.
Finally, Theorem \ref{cor:frobenian} is a formal consequence of Lemma~\ref{lem:s=1_density} and Theorem~\ref{thm:CT}, because the intersection of finitely many frobenian sets is frobenian (see \cite[\S 3.3.1]{Ser12}). \qed

\medskip

\begin{remark} \label{remark:Denef} Let us finish by explaining how to recover Denef's result (Theorem~\ref{deneftheorem}) from our Theorem~\ref{maintheorem}. The subtlety is the following. Denef imposes a condition on modifications $f': X' \to Y'$ of $f$ such that the generic fibres of $f$ and $f'$ are isomorphic. We consider \emph{arbitrary} modifications $f': X' \to Y'$ of $f$, since using Theorem  \ref{ADK} forces us to consider modifications for which the generic fibre changes birationally, and we impose a \emph{weaker} condition for these modifications (pseudo-splitness instead of splitness). 

But in fact we could just as well have imposed our pseudo-splitness condition only for the modifications $f': X' \to Y'$ with generic fibre isomorphic to the generic fibre of $f$, as does Denef. It turns out that this is enough to guarantee that the same condition holds for arbitrary modifications, as we will now explain.

The argument is the following. Let $f: X \to Y$ be a dominant morphism of smooth, proper, geometrically integral varieties over $k$, with geometrically integral generic fibre. Assume that for every modification $f': X' \to Y'$ of $f$ with $X'$ and $Y'$ smooth
such that the generic fibres of $f$ and $f'$ are isomorphic and for every $D~\in~(Y')^{(1)}$, the fibre $(f')^{-1}(D)$ is a pseudo-split $\kappa(D)$-variety. Let us check that the same property then holds for an arbitrary modification $f': X' \to Y'$ of $f$ with $X'$ and $Y'$ smooth. 

So let $f': X' \to Y'$ be such an arbitrary modification. Let $Z$ be the unique irreducible component of $X \times_Y Y'$ which dominates $X$, equipped with the natural morphism $Z \to Y'$. Note that the generic fibre of $Z \to Y'$ is the same as the generic fibre of $f$. Choose a desingularisation $\widetilde{Z} \to Z$ such that the composition
$\widetilde{Z} \to Y'$ still has the same generic fibre as $f$. Let $D \in (Y')^{(1)}$ and let $R = \OO_{Y',D}$
be the corresponding discrete valuation ring. Consider the $R$-scheme $\widetilde{Z} \times_{Y'} \Spec R$. This is a regular scheme
with the same generic fibre as $f$. The assumptions imply that its special fibre is pseudo-split.

However the $R$-scheme $X' \times_{Y'} \Spec R$
is regular and its generic fibre is smooth and birational to the generic fibre of $f$. Therefore Lemma \ref{birationality} implies
that its special fibre is pseudo-split as well, which proves our claim.
\end{remark}

\begin{remark} \label{remark:Liang} Yongqi Liang observed that exactly the same proofs yield in fact a slightly stronger version of Theorem \ref{maintheorem}. Indeed, the conclusion that $f(X(k_v)) = Y(k_v)$ for almost all places $v$ (assuming the pseudo-splitness assumption) may be replaced by the following stronger conclusion: there exists a finite set of places $S$ of $k$ such that if $K/k$ is a finite field extension and if $w$ is a place of $K$ which does not lie above any of the places in $S$, then $f(X(K_w)) = Y(K_w)$. \end{remark}

\end{document}